\documentclass{article}
\usepackage[utf8]{inputenc}
\usepackage{amssymb,amsthm,bussproofs}
\usepackage{MnSymbol}
\usepackage{hyperref,url}
\usepackage{bussproofs}
\usepackage{xcolor}

\EnableBpAbbreviations

\newtheorem{theorem}{Theorem}
\newtheorem{lemma}{Lemma}
\newtheorem{corollary}{Corollary}
\newtheorem{proposition}{Proposition}

\DeclareMathOperator {\sat}{\mathsf{sat}}
\DeclareMathOperator {\sz}{\textit{sz}\,}

\title{A realization theorem for the modal logic of transitive closure $\mathsf{K}^+$}
\author{Daniyar Shamkanov\thanks{The article was prepared within the framework of the project ``International academic cooperation" HSE University.}\\ \normalsize{\textit{Steklov Mathematical Institute of the Russian Academy of Sciences}}\\
\normalsize{\textit{National Research University Higher School of Economics}}\\
\normalsize{\textit{daniyar.shamkanov@gmail.com}}\\}
\date{}

\begin{document}

\maketitle

\begin{abstract}
We present a justification logic corresponding to the modal logic of transitive closure $\mathsf{K}^+$ and establish a normal realization theorem relating these two systems. The result is obtained by means of a sequent calculus allowing non-well-founded proofs.\\\\
\textit{Keywords:} justification logic, transitive closure, realization theorems,  cyclic and non-well-founded proofs.
\end{abstract}

\section{Introduction}
It is worth to recall that justification logics are a family of epistemic systems whose language feature is the replacement of modal expressions $\Box A$ with $[h]\, A$, where $[h]\, A$ is interpreted as `$h$ is a justification for $A$'. A lot of research on justification logics has been undertaken since Artemov introduced the logic of proofs $\mathsf{LP}$ \cite{Artemov2001}, where $[h] A$ is understood as ‘$h$ is a proof for $ A$’. Among other things established by Artemov, a realization theorem relating $\mathsf{LP}$ and the standard modal logic $\mathsf{S4}$ has attracted considerable attention. This result involves the following notion of forgetful translation (or projection). For any formula $B$ of $\mathsf{LP}$, the forgetful translation of $B$ is obtained from the given formula by replacing all subformulas of the form $[h]\, C$ with $\Box C$. It is easy to see that the forgetful translation of every formula provable in $\mathsf{LP}$ is provable in $\mathsf{S4}$. The realization theorem states the converse: any formula $A$ provable in $\mathsf{S4}$ turns out to be the forgetful translation of a formula $B$ provable in $\mathsf{LP}$. Due to the theorem, the logic $\mathsf{LP}$ is called a justification counterpart of $\mathsf{S4}$. 
 


To date, justification counterparts of many modal logics have been found and corresponding realization theorems have been obtained. However, the epistemically important case of the modal logic of common knowledge still needs to be explored. The concept of common knowledge is captured in this logic according to the so-called fixed-point account, i.e. common knowledge of $A$ is defined as the greatest fixed-point of the mapping 
\[X \mapsto (\text{everybody knows $A$ and everybody knows $X$)}.\] 
Accordingly, the logic of common knowledge belongs to the family of modal fixed-point logics and, like other systems from this family, is difficult to study in many respects. Although Bucheli, Kuznets and Studer introduced a justification logic similar to the logic of common knowledge, whether one can prove the realization theorem remains to be an open question  \cite{Bucheli2011, Bucheli2012}. 
Note that, for the modified concept of common knowledge known as generic common knowledge, the corresponding modal logic turns out to be realizable \cite{Antonakos2013,Artemov2006}.

This article focuses on the case of the logic of transitive closure $\mathsf{K}^+$, which is very similar to the case of the modal logic of common knowledge. We recall that the system $\mathsf{K}^+$ \cite{Kashima2010,Doczkal2012,Kikot2020} is a Kripke-complete modal propositional logic whose language contains modal connectives $\Box$ and $\Box^+$. Like the modal logic of common knowledge, this system belongs to the family of modal fixed-point logics, which can be explained as follows: in any $\mathsf{K^+}$-algebra $\mathcal{A}$, an element $\Box^+ a$ is the greatest fixed-point of a mapping $z \mapsto \Box a \wedge \Box z$, i.e. $\Box^+ a = \nu z. (\Box a \wedge \Box z)$. Therefore, it is not surprising that $\mathsf{K^+}$ is not valid in its canonical Kripke frame and is not strongly complete with respect to its Kripke semantics. In the given work, we present a justification counterpart of $\mathsf{K}^+$ and establish the corresponding realization theorem by means of a sequent calculus allowing non-well-founded proof trees. It remains to emphasize that we know only one more theorem about normal realization for a logic that is not Kripke-canonical, namely,  a realization theorem for the G\"{o}del-L\"{o}b provability logic $\mathsf{GL}$ (see \cite{Shamkanov2016}).

\section{Logics $\mathsf{K}^+$ and $\mathsf{J}^+$} 
In this section, we briefly remind the reader of the bimodal logic $\mathsf{K}^+$ \cite{Kashima2010,Doczkal2012} and define a justification logic $\mathsf{J}^+$, which will be proved to be a counterpart of $\mathsf{K}^+$. We also prove some properties of $\mathsf{J}^+$ in order to use them later.

\textit{Formulas of $\mathsf{K}^+$} are built from propositional variables $p_0, p_1, p_2, \dotsc$ and the constant $\bot$ by means of propositional connectives $\to$, $\Box$ and $\Box^+$. We consider other Boolean connectives as abbreviations: $\neg A  := A\rightarrow \bot$, $\top  := \neg \bot$, $A \wedge B  := \neg (A \rightarrow \neg B)$, $A \vee B  := (\neg A \rightarrow B)$. The \emph{size of a formula $A$}, denoted by $\sz (A)$, is defined inductively in the following way:
\begin{gather*}
\sz(p):=1, \qquad \sz(\bot):=1, \qquad \sz (A\to B):= \sz(A)+\sz(B)+1,\\
\sz(\Box A):=\sz(A)+1,\qquad \sz(\Box^+ A):=\sz(A)+1.
\end{gather*}


The Frege-Hilbert calculus of the logic $\mathsf{K}^+$ is given by the following axioms and inference rules.\medskip

\textit{Axioms:}
\begin{itemize}
\item $A \to (B \to A)$;
\item $(A \to (B \to C)) \to ((A \to B) \to (A \to C)) $;
\item $\neg
\neg A \to A$;
\item $\Box (A \rightarrow B) \rightarrow (\Box A \rightarrow \Box B)$;
\item $\Box^+ (A \rightarrow B) \rightarrow (\Box^+ A \rightarrow \Box^+ B)$;
\item $\Box^+ A \rightarrow \Box A $;
\item $\Box^+ A \rightarrow \Box \Box^+ A$;
\item $\Box A \wedge \Box^+(A \rightarrow \Box A) \rightarrow \Box^+ A$.
\end{itemize}

\textit{Inference rules:} 
\[
\AXC{$A$}
\AXC{$A \rightarrow B$}
\LeftLabel{\textsf{mp}}
\RightLabel{ ,}
\BIC{$B$}
\DisplayProof \qquad
\AXC{$A$}
\LeftLabel{$\mathsf{nec}$}
\RightLabel{ .}
\UIC{$\Box^+ A$}
\DisplayProof 
\]


Recall that a bimodal Kripke frame $(W,R,S)$ is a \emph{$\mathsf{K}^+$-frame} if the relation $S$ is the transitive closure of $R$.

\begin{proposition}[see \cite{Kashima2010,Doczkal2012,Kikot2020}]
The logic $\mathsf{K}^+$ is sound and weakly complete with respect to the class of $\mathsf{K}^+$-frame. 
\end{proposition}

Now we define a justification logic $\mathsf{J}^+$. The language of  $\mathsf{J}^+$ contains three sorts of expressions: two sorts of terms and one sort of formulas.
\emph{Justification terms} of both sorts are simultaneously built from the disjoint countable sets of variables $\mathit{JV}_1 =\{x_0, x_1, \dotsc \}$ and $\mathit{JV}_2 =\{y_0, y_1, \dotsc \}$ and constants $\mathit{JC}_2 =\{c_0, c_1, \dotsc \}$
according to the grammar:
\begin{gather*}
w ::= x_i \,\,|\,\, (w \cdot w) \,\,|\,\,  \mathsf{head}(s)  \,\,|\,\,\mathsf{tail}(s)  \,\,|\,\, (w+w) ,\\
s ::= y_i \,\,|\,\, c_i \,\,|\,\, (s \cdot s) \,\,|\,\,  \mathsf{ind}(w,s)   \,\,|\,\, (s+s) ,  
\end{gather*}
where $w$ and $s$ stand for justification terms of the first and second sort respectively. The corresponding sets of terms are denoted by $\mathit{JT}_1$ and $\mathit{JT}_2$.
We call a justification term \emph{ground} if it doesn't contain variables.
\emph{Justification formulas} are given by the grammar:
\[ A ::=  p_i \,\,|\,\, \bot \,\,|\,\, (A \rightarrow A)  \,\,|\,\, [w] A \,\,|\,\,  [ s]_\mathsf{tc} \,A . \]  
We denote the set of justification formulas by $\mathit{JF}$.    

The logic $\mathsf{J}^+_0$ is defined by the following axioms and the following inference rule.\medskip

\textit{Axioms:}
\begin{itemize}
\item[(i)] $A \to (B \to A)$;
\item[(ii)] $(A \to (B \to C)) \to ((A \to B) \to (A \to C)) $;
\item[(iii)] $\neg
\neg A \to A$;
\item[(iv)] $[h] (A \rightarrow B) \rightarrow ([w]A \rightarrow [h \cdot w] B)$;
\item[(v)] $[h]A \vee [w] A \rightarrow [h+w] A$;
\item[(vi)] $[ t]_\mathsf{tc}\, (A \rightarrow B) \rightarrow ([s]_\mathsf{tc}\, A \rightarrow [ t \cdot s ]_\mathsf{tc}\, B)$;
\item[(vii)] $[ s]_\mathsf{tc}\, A \rightarrow [\mathsf{head}(s)] A$;
\item[(viii)] $[ s]_\mathsf{tc}\, A \rightarrow [\mathsf{tail}(s)]  [ s]_\mathsf{tc}\, A$;
\item[(ix)] $ [w] A \wedge [s]_\mathsf{tc}\, (A \rightarrow [w] A) \rightarrow  [\mathsf{ind}(w,s)]_\mathsf{tc} \, A$;
\item[(x)] $[t]_\mathsf{tc}\,A \vee [s]_\mathsf{tc}\, A \rightarrow [t+s]_\mathsf{tc}\, A$.
\end{itemize}

\textit{Inference rule:} 
\begin{gather*}
\AXC{$A$}
\AXC{$A\to B$}
\LeftLabel{$\mathsf{mp}$}
\RightLabel{ .}
\BIC{$B$}
\DisplayProof\qquad
\end{gather*}\smallskip

We introduce the logic $\mathsf{J}^+$ by adding the following set of new axioms
\[\{ [c ]_\mathsf{tc} \,A \mid c \in \mathit{JC}_2 \text{ and $A$ is an axiom of $\mathsf{J}^+_0$} \}\]
to $\mathsf{J}^+_0$.
Subsets of the given set of axioms are called \emph{constant specifications}. For a constant specification $\mathit{cs}$, let $\mathsf{J}^+_\mathit{cs}$ be the fragment of $\mathsf{J}^+$ in which all axioms of the form $[c ]_\mathsf{tc} \,A $ are taken from $\mathit{cs}$. Note that $\mathsf{J}^+_\emptyset$ is the same as $\mathsf{J}^+_0$. Additionally, we define the set of constants $\mathit{Con}(\mathit{cs})$ by setting $c\in \mathit{Con}(\mathit{cs})$ if and only if $[c ]_\mathsf{tc} \,A$ belongs to $\mathit{cs} $ for some formula $A$.

A constant specification $\mathit{cs}$ is called \emph{injective} if, for any $[c_i ]_\mathsf{tc} \,A$ and $[c_i ]_\mathsf{tc} \,B$ from $\mathit{cs}$, the formulas $A$ and $B$ coincide. In other words, different axioms of $\mathsf{J}^+_0$ are associated with different constants in $\mathit{cs}$. For a proof $\pi$ of $\mathsf{J}^+$, we denote the set of all axioms of the form $[c ]_\mathsf{tc} \,A$ in $\pi$ by $\mathit{cs}(\pi)$. The proof $\pi$ is called \emph{injective} if the constant specification $\mathit{cs}(\pi)$ is injective. Note that any proof of $\mathsf{J}^+_\mathit{cs}$, where $\mathit{cs}$ is injective, is also injective. Further note that $\sigma(\mathit{cs})= \{\sigma (C)\mid C\in \mathit{cs}\} $ is an injective constant specification for any injective constant specification $\mathit{cs}$ and any substitution 
\[\sigma =[h_1/x_{i_1}, \dotsc , h_n/x_{i_n}, t_1/y_{j_1}, \dotsc , t_m/y_{j_m}],\] 
where $h_1,\dotsc , h_n$ ($t_1,\dotsc , t_m$) are justification terms of the first (second) sort.




\begin{lemma}[substitution] 
\label{Substitution lemma} If $\mathsf{J}^+_\mathit{cs} \vdash A$, then, for any substitution
\[\sigma =[h_1/x_{i_1}, \dotsc , h_n/x_{i_n}, t_1/y_{j_1}, \dotsc , t_m/y_{j_m}],\] 
we have $\mathsf{J}^+_{\sigma(\mathit{cs})} \vdash \sigma(A)$. In particular, if $A$ has an injective proof in $\mathsf{J}^+ $, then so does $\sigma (A)$.
\end{lemma}
\begin{proof}
The assumption $\mathsf{J}^+_\mathit{cs} \vdash A$ immediately implies $\mathsf{J}^+_{\sigma(\mathit{cs})} \vdash \sigma(A)$. In addition, if $A$ has an injective proof $\pi$ in $\mathsf{J}^+ $, then $\mathsf{J}^+_{\mathit{cs}(\pi)} \vdash A$ and the constant specification $\mathit{cs}(\pi)$ is injective. Therefore, $\sigma (\mathit{cs}(\pi))$ is injective. Since $\mathsf{J}^+_{\sigma(\mathit{cs})} \vdash \sigma(A)$, the formula $\sigma (A)$ has an injective proof in $\mathsf{J}^+_{\sigma(\mathit{cs})}$ and in $\mathsf{J}^+ $.
\end{proof}

\begin{lemma}[axiom internalization] Suppose $\mathsf{J}^+_{\mathit{cs}_0} \vdash A$, where $\mathit{cs}_0$ is a finite constant specification and $A$ is an axiom of $\mathsf{J}^+$. Then there exist a finite superset $\mathit{cs}_1$ of $\mathit{cs}_0$ and a ground justification term $s$ such that $\mathsf{J}^+_{\mathit{cs}_1} \vdash [s]_\mathsf{tc} \, A$. Moreover, if $\mathit{cs}_0$ is injective, then the same holds for $\mathit{cs}_1$.
\end{lemma}
\begin{proof}
If $A$ is an axiom of $\mathsf{J}^+_0$, then $\mathsf{J}^+_{\mathit{cs}_1} \vdash [c_i]_\mathsf{tc} \, A$, where $c_i$ is the first justification constant not belonging to $\mathit{Con}(\mathit{cs}_0)$ and $\mathit{cs}_1\coloneq \mathit{cs}_0\cup \{[c_i]_\mathsf{tc} \, A\}$. If $A$ has the form $[c]_\mathsf{tc} \, B$, then $\mathsf{J}^+_{\mathit{cs}_0} \vdash [ c]_\mathsf{tc}\, B \rightarrow [\mathsf{tail}(c)]  [ c]_\mathsf{tc}\, B$. In this case, $\mathsf{J}^+_{\mathit{cs}_1} \vdash [c_i]_\mathsf{tc} \, ([ c]_\mathsf{tc}\, B \rightarrow [\mathsf{tail}(c)]  [ c]_\mathsf{tc}\, B)$, where $c_i$ is the first justification constant not belonging to $\mathit{Con}(\mathit{cs}_0)$ and $\mathit{cs}_1\coloneq \mathit{cs}_0\cup \{[c_i]_\mathsf{tc} \, ([ c]_\mathsf{tc}\, B \rightarrow [\mathsf{tail}(c)]  [ c]_\mathsf{tc}\, B)\}$. Since $\mathsf{J}^+_{\mathit{cs}_0} \vdash [ c]_\mathsf{tc}\, B$, we have $\mathsf{J}^+_{\mathit{cs}_1} \vdash [ c]_\mathsf{tc}\, B$ and $\mathsf{J}^+_{\mathit{cs}_1} \vdash  [\mathsf{tail}(c)]  [ c]_\mathsf{tc}\, B$. Therefore, $\mathsf{J}^+_{\mathit{cs}_1} \vdash [\mathsf{tail}(c)]  [ c]_\mathsf{tc}\, B\wedge [c_i]_\mathsf{tc} \, ([ c]_\mathsf{tc}\, B \rightarrow [\mathsf{tail}(c)]  [ c]_\mathsf{tc}\, B)$. Applying Axiom (ix), we obtain $\mathsf{J}^+_{\mathit{cs}_1} \vdash [\mathsf{ind}(\mathsf{tail}(c),c_i)]_\mathsf{tc}\, [ c]_\mathsf{tc}\, B$, i.e. $\mathsf{J}^+_{\mathit{cs}_1} \vdash [\mathsf{ind}(\mathsf{tail}(c),c_i)]_\mathsf{tc} \,A$. Trivially, in both cases, $\mathit{cs}_1$ is injective if $\mathit{cs}_0$ is injective.
\end{proof}
\begin{lemma}[internalization]\label{internalization} Suppose $\mathsf{J}^+_{\mathit{cs}_0} \vdash A$, where $\mathit{cs}_0$ is a finite constant specification. Then there exist a finite superset $\mathit{cs}_1$ of $\mathit{cs}_0$ and a ground justification term $s$ such that $\mathsf{J}^+_{\mathit{cs}_1} \vdash [s]_\mathsf{tc} \, A$. Moreover, if $\mathit{cs}_0$ is injective, then the same holds for $\mathit{cs}_1$.
\end{lemma}
\begin{proof}
Assume $\mathsf{J}^+_{\mathit{cs}_0} \vdash A$ and consider a proof $\pi$ of $A$ in $\mathsf{J}^+_{\mathit{cs}_0} $. Let $B_1, \dotsc, B_n$ be the axioms of $\mathsf{J}^+_{\mathit{cs}_0} $ that mark the leaves of $\pi$. Successively applying the previous lemma to the formulas $B_1, \dotsc, B_n$ and expanding the resulting constant specifications, we find a finite superset $\mathit{cs}_1$ of $\mathit{cs}_0$ and ground justification terms $s_1,\dotsc, s_n$ such that $\mathsf{J}^+_{\mathit{cs}_1} \vdash [s_i]_\mathsf{tc} \, B_i$ for $i\in\{1,\dotsc, n\}$. Moreover, $\mathit{cs}_1$ is injective if $\mathit{cs}_0$ is injective. Notice that $\pi$ is a tree whose leaves are marked by axioms $B_1, \dotsc, B_n$ and that is constructed according to the rule ($\mathsf{mp}$). Consequently, moving from the leaves of $\pi$ to its root and applying Axiom (vi), we can find, for each node $b$, a ground justification term $s_b$ such that $\mathsf{J}^+_{\mathit{cs}_1} \vdash [s_b]_\mathsf{tc} \, C_b$, where $C_b$ is the formula of the node $b$. Therefore, there is a ground justification term $s$ such that $\mathsf{J}^+_{\mathit{cs}_1} \vdash [s]_\mathsf{tc} \, A$.
\end{proof}
\begin{lemma}[lifting lemma] 
\label{Lifting lemma} Suppose
\[\mathsf{J}^+ \vdash A_1\wedge\dotso \wedge A_n\wedge B_1\wedge\dotso \wedge  B_m
\wedge [s_1]_\mathsf{tc}\, B_1\wedge\dotso \wedge [s_m]_\mathsf{tc}\, B_m \to C.\]
Then there exists a justification term $h(z_1, \dotsc , z_n, y_1, \dotsc , y_m)$ depending only on the explicitly displayed variables such that
\[\mathsf{J}^+ \vdash [z_1]A_1\wedge\dotso \wedge [z_n]A_n\wedge [s_1]_\mathsf{tc}\, B_1\wedge\dotso \wedge [s_m]_\mathsf{tc}\, B_m \to [h(z_1, \dotsc , z_n, s_1, \dotsc , s_m)]C\]
for arbitrary variables $z_1, \dotsc, z_n$ of the first sort.
Moreover, if the original proof is injective, the same holds for the later proof.
\end{lemma}
\begin{proof}
Assume 
\[ \mathsf{J}^+ \vdash A_1\wedge\dotso \wedge A_n\wedge B_1\wedge\dotso \wedge  B_m
\wedge [s_1]_\mathsf{tc}\, B_1\wedge\dotso \wedge [s_m]_\mathsf{tc}\, B_m \to C.\]
Then this formula is provable in $\mathsf{J}^+_{\mathit{cs}_0}$ for some finite constant specification $\mathit{cs}_0$. Therefore, 
\[\mathsf{J}^+_{\mathit{cs}_0} \vdash A_1\to (A_2 \to \dotso (B_1\to  \dotso ([s_1]_\mathsf{tc}\, B_1\to \dotso ([s_m]_\mathsf{tc}\, B_m \to C)\dots).\]
From Lemma \ref{internalization}, there exist a finite superset $\mathit{cs}_1$ of $\mathit{cs}_0$ and a ground justification term $t$ such that
\[\mathsf{J}^+_{\mathit{cs}_1} \vdash [t]_\mathsf{tc} \, (A_1\to (A_2 \to \dotso (B_1\to  \dotso ([s_1]_\mathsf{tc}\, B_1\to \dotso ([s_m]_\mathsf{tc}\, B_m \to C)\dots)).\]
From Axiom (vii), it follows that
\[\mathsf{J}^+_{\mathit{cs}_1} \vdash [\mathsf{head}(t)] (A_1\to (A_2 \to \dotso (B_1\to  \dotso ([s_1]_\mathsf{tc}\, B_1\to \dotso ([s_m]_\mathsf{tc}\, B_m \to C)\dots)).\]
Applying Axiom (iv) successively, we obtain
\begin{multline*}
\mathsf{J}^+_{\mathit{cs}_1} \vdash [z_1]A_1\to ([z_2]A_2 \to \dotso ([\mathsf{head}(s_1)]B_1\to  \dotso ([\mathsf{tail}(s_1)][s_1]_\mathsf{tc}\, B_1\to \dotso \\([\mathsf{tail}(s_m)][s_m]_\mathsf{tc}\, B_m \to [h(z_1, \dotsc , z_n, s_1, \dotsc , s_m)] C)\dots),
\end{multline*}
where $h(z_1, \dotsc , z_n, y_1, \dotsc , y_m)$ is equal to
\[(\dots(\mathsf{head}(t)\cdot z_1) \cdot \dotso \cdot z_n) \cdot \mathsf{head}(y_1)) \cdot \dotso \cdot \mathsf{head}(y_m))\cdot \mathsf{tail}(y_1)) \cdot \dotso \cdot \mathsf{tail}(y_m).\]
Hence, the formula
\begin{multline*}
[z_1]A_1\wedge\dotso \wedge [z_n]A_n\wedge [\mathsf{head}(s_1)]B_1\wedge\dotso \wedge  [\mathsf{head}(s_m)]B_m
\wedge [\mathsf{tail}(s_1)][s_1]_\mathsf{tc}\, B_1\wedge\dotso \\
\wedge [\mathsf{tail}(s_m)][s_m]_\mathsf{tc}\, B_m \to [h(z_1, \dotsc , z_n, s_1, \dotsc , s_m)]C
\end{multline*}
is provable in $\mathsf{J}^+_{\mathit{cs}_1}$.
Applying Axiom (vii) and Axiom (viii), we conclude  
\[\mathsf{J}^+_{\mathit{cs}_1} \vdash [z_1]A_1\wedge\dotso \wedge [z_n]A_n\wedge [s_1]_\mathsf{tc}\, B_1\wedge\dotso \wedge [s_m]_\mathsf{tc}\, B_m \to [h(z_1, \dotsc , z_n, s_1, \dotsc , s_m)]C.\]
Note that the constant specifications $\mathit{cs}_0$ and $\mathit{cs}_1$ can be chosen to be injective if the original proof was injective. 
\end{proof}
\begin{lemma}\label{induction rule lemma}
If $\mathsf{J}^+ \vdash B \to [w] (A \wedge B)$, then there exists a justification term $t (x_0)$ depending only on $x_0$ such that $\mathsf{J}^+ \vdash B \to [t(w)]_\mathsf{tc} \, A$. Moreover, if the original proof is injective, the same holds for the later proof. 
\end{lemma}

\begin{proof}
Assume $\mathsf{J}^+ \vdash B \to [w] (A \wedge B)$. Then $\mathsf{J}^+_{\mathit{cs}_0} \vdash B \to [w] (A \wedge B)$ for some finite constant specification $\mathit{cs}_0$. We have $\mathsf{J}^+_{\mathit{cs}_0} \vdash A \wedge B \to [w] (A \wedge B)$. By Lemma \ref{internalization}, there are a finite superset ${\mathit{cs}_1}$ of ${\mathit{cs}_0}$ and a ground term $s_0$ such that $\mathsf{J}^+_{\mathit{cs}_1} \vdash  [s_0]_\mathsf{tc} \, (A \wedge B \to [w] (A \wedge B))$. Hence, $\mathsf{J}^+_{\mathit{cs}_1} \vdash B \to ([w] (A \wedge B) \wedge [s_0]_\mathsf{tc} \, (A \wedge B \to [w] (A \wedge B)))$. Applying Axiom (ix), we obtain $\mathsf{J}^+_{\mathit{cs}_1} \vdash B \to [\mathsf{ind}(w,s_0)]_\mathsf{tc} \,  (A \wedge B))$. Besides, $\mathsf{J}^+_{\mathit{cs}_1} \vdash A \wedge B \to A$. From Lemma \ref{internalization}, there are a finite superset ${\mathit{cs}_2}$ of ${\mathit{cs}_1}$ and a ground term $s_1$ such that $\mathsf{J}^+_{\mathit{cs}_2} \vdash [s_1]_\mathsf{tc} \,  (A \wedge B \to A)$. Applying Axiom (vi), we obtain $\mathsf{J}^+_{\mathit{cs}_2} \vdash B \to [s_1 \cdot \mathsf{ind}(w,s_0)]_\mathsf{tc} \, A$. It remains to note that $\mathit{cs}_2$ can be chosen to be injective if $\mathit{cs}_0$ is injective.
\end{proof}

\section{A non-well-founded sequent calculus}
This section examines a sequent calculus for the logic $\mathsf{K}^+$, where non-well-founded proofs are allowed. The given system, denoted by $\mathsf{S}$, is a version of the calculus from \cite{Bucheli2010} adapted for the case of transitive closure. Below we provide a semantic proof that each theorem of $\mathsf{K}^+$ is provable in $\mathsf{S}$. We present the argument in full detail, although semantic proofs of the given sort are not new (see \cite{Bucheli2010} and \cite{Doczkal2012}). We also stress that the established connection between two calculi can be proved in a pure syntactic way (see Section 8 of \cite{Shamkanov2023}).

\textit{Sequents} are defined as expressions of the form $\Gamma \Rightarrow \Delta$, where $\Gamma$ and $\Delta$ are finite multisets of formulas. Multisets are often written without any curly braces, and the comma in the expression $\Gamma, \Delta$ means the multiset union. For a multiset of formulas $\Gamma = A_1,\dotsc, A_n$, we put $\Box \Gamma := \Box A_1,\dotsc, \Box A_n$ and $\Box^+ \Gamma := \Box^+ A_1,\dotsc, \Box^+ A_n$. If we remove all repetitions in a multiset $\Gamma$, then the resulting multiset is denoted by $\Gamma^s$.
For example, $\Gamma^s= p,q, \Box^+ (p \to q)$ if $\Gamma= p,p, p,q, \Box^+ (p \to q), \Box^+ (p \to q)$.


We denote the sequent calculus for the logic $\mathsf{K}^+$ by $\mathsf{S}$ and define its inference rules as follows:
\begin{gather*}
\AXC{}
\RightLabel{ ,}
\UIC{$\Gamma, p \Rightarrow p, \Delta $}
\DP  \qquad
\AXC{}
\RightLabel{ ,}
\UIC{$\Gamma, \bot \Rightarrow  \Delta $}
\DP 
\end{gather*}
\begin{gather*}
\AXC{$\Gamma , B \Rightarrow  \Delta$}
\AXC{$\Gamma \Rightarrow  A, \Delta$}
\LeftLabel{$\mathsf{\rightarrow_L}$}
\RightLabel{ ,}
\BIC{$\Gamma , A \rightarrow B \Rightarrow  \Delta$}
\DisplayProof \qquad
\AXC{$\Gamma, A \Rightarrow  B ,\Delta$}
\LeftLabel{$\mathsf{\rightarrow_R}$}
\RightLabel{ ,}
\UIC{$\Gamma \Rightarrow  A \rightarrow B ,\Delta$}
\DisplayProof 
\end{gather*}
\begin{gather*}
\AXC{$\Sigma, \Pi, \Box^+ \Pi \Rightarrow A$}
\LeftLabel{$\mathsf{\Box}$}
\RightLabel{ ,}
\UIC{$ \Upsilon, \Box \Sigma, \Box^+ \Pi \Rightarrow \Box A , \Lambda$}
\DisplayProof \qquad
\AXC{$\Sigma, \Pi, \Box^+ \Pi \Rightarrow A$}
\AXC{$\Sigma, \Pi, \Box^+ \Pi \Rightarrow \Box^+ A$}
\LeftLabel{$\Box^+$}
\RightLabel{ .}
\BIC{$ \Upsilon, \Box \Sigma, \Box^+ \Pi \Rightarrow  \Box^+ A ,\Lambda$}
\DisplayProof 
\end{gather*}
The last two inference rules of the sequent calculus are called \emph{modal rules}. 
For the modal rule ($\Box$) (or ($\Box^+$)), 
the formula $\Box A$ (or $\Box^+ A$) is the \emph{principal formula} of the corresponding inference. 


An \emph{$\infty$-proof} is a possibly infinite tree whose nodes are marked by sequents and that is constructed according to the rules of the sequent calculus. Besides, 
every infinite branch in an $\infty$-proof must contain a tail satisfying the conditions: all applications of the rule ($\mathsf{\Box^+}$) in the tail have the same principal formula $\Box^+ A$; the tail passes through the right premise of the rule ($\mathsf{\Box^+}$) infinitely many times; the tail doesn't pass through the left premise of the rule ($\mathsf{\Box^+}$); there are no applications of the rule ($\Box$) in the tail. 

An $\infty$-proof is called \emph{regular} if it contains only finitely many non-isomorphic subtrees with respect to the marking of sequents.
A sequent $\Gamma \Rightarrow \Delta$ is \emph{provable in $\mathsf{S}$} if there is a regular $\infty$-proof $\pi$ with the root marked by $\Gamma \Rightarrow \Delta$. 

For example, consider the regular $\infty$-proof
\begin{gather*}\label{Example}
\AXC{$p , H, \Box^+ H  \Rightarrow   p$}
\AXC{$\pi$}
\noLine
\UIC{$\vdots$}
\noLine
\UIC{$p ,  \Box p, \Box^+ H  \Rightarrow   \Box^+ p$}
\AXC{$p ,   \Box^+ H  \Rightarrow    p, \Box^+ p$}
\LeftLabel{$\mathsf{\rightarrow_L}$}
\BIC{$p , H, \Box^+ H  \Rightarrow   \Box^+ p$}
\LeftLabel{$\mathsf{\Box}^+$}
\RightLabel{ ,}
\BIC{$p, \Box p , \Box^+ H  \Rightarrow  \Box^+ p$}
\DisplayProof 
\end{gather*}
where $H= p \rightarrow \Box p$ and the subtree $\pi$ is isomorphic to the whole $\infty$-proof. Here the unique infinite branch passes through alternate applications of inference rules ($\mathsf{\rightarrow_L}$) and ($\mathsf{\Box^+}$) infinitely many times. If we consider the given branch as its own tail, then we immediately see that this branch satisfies the required conditions on infinite branches in $\infty$-proofs.  
 
We call a sequent $\Gamma \Rightarrow \Delta$ \emph{valid} if the formula $\bigwedge\Gamma\to \bigvee\Delta$ is valid in any bimodal Kripke frame $(W,R,R^+)$, where $R^+$ is the transitive closure of $R$. In the rest of the section, we show that any valid sequent is provable in $\mathsf{S}$.

Let us consider the following auxiliary rules ($\mathsf{\to_{L1}}$) and ($\mathsf{\to_{L2}}$)  
\begin{gather*}
\AXC{$\Gamma , B \Rightarrow  \Delta$}
\LeftLabel{$\mathsf{\rightarrow_{L1}}$}
\RightLabel{ ,}
\UIC{$\Gamma , A \rightarrow B \Rightarrow  \Delta$}
\DisplayProof \qquad
\AXC{$\Phi \Rightarrow  C, \Psi$}
\LeftLabel{$\mathsf{\rightarrow_{L2}}$}
\RightLabel{}
\UIC{$\Phi , C \rightarrow D \Rightarrow  \Psi$}
\DisplayProof 
\end{gather*}
with the side conditions: the sequent $\Gamma \Rightarrow  A, \Delta$ (the sequent $\Phi , D \Rightarrow   \Psi$) is provable in $\mathsf{S}$.

Furthermore, we consider the rule ($\boxtimes$) 
\begin{gather*}
\AXC{$\textsc{Prem}_1$}
\AXC{$\textsc{Prem}_2$}
\LeftLabel{$\boxtimes$}
\RightLabel{ ,}
\BIC{$ \Upsilon, \Box \Sigma, \Box^+ \Pi \Rightarrow \Box A_1, \dotsc\Box A_n, \Box^+ B_1 \dotsc \Box^+ B_m,\Lambda$}
\DisplayProof 
\end{gather*}
where the multisets $\Upsilon$ and $\Lambda$ contain only propositional variables and the constant $\bot$. Besides, $\textsc{Prem}_1$ and $\textsc{Prem}_2$ are two (possibly empty) groups of premises such that
\begin{itemize}
    \item $ \textsc{Prem}_1$ contains $\Sigma^s, \Pi^s, \Box^+ \Pi^s \Rightarrow A_i$ for each $1\leqslant i \leqslant n$,
    \item $\textsc{Prem}_2$ contains one or both of the sequents 
    \[\Sigma^s, \Pi^s, \Box^+ \Pi^s \Rightarrow B_j, \qquad \Sigma^s, \Pi^s, \Box^+ \Pi^s \Rightarrow \Box^+ B_j\] for each $1\leqslant j\leqslant m$.
\end{itemize}
In addition, the rule ($\boxtimes$) has the side condition: any sequent of the form $\Sigma^s, \Pi^s, \Box^+ \Pi^s \Rightarrow B_j$ or $\Sigma^s, \Pi^s, \Box^+ \Pi^s \Rightarrow \Box^+ B_j$ that doesn't belong to $\textsc{Prem}_2$ is provable in $\mathsf{S}$.

A sequent $\Gamma \Rightarrow \Delta$ is called \emph{saturated} if $\Gamma$ and $\Delta$ do not contain formulas of the form $A \to B$. 
A \emph{saturation tree} 
is a finite tree of unprovable sequents constructed according to the rules ($\mathsf{\to_{R}}$), ($\mathsf{\to_{L}}$), ($\mathsf{\to_{L1}}$) and ($\mathsf{\to_{L2}}$),
where 
all leaves are marked by saturated sequents.
\begin{lemma}\label{Saturation lemma}
For any unprovable sequent $\Gamma \Rightarrow \Delta$, there is a saturation tree with the root marked by $\Gamma \Rightarrow \Delta$.  
\end{lemma}
\begin{proof}
For a sequent $\Phi \Rightarrow \Psi$, we define its size as the sum of sizes of all formulas from $\Phi$ and $\Psi$ with respect to repetitions. 

Now assume we have an unprovable sequent $\Gamma \Rightarrow \Delta$. We prove that there exists the required saturation tree for $\Gamma \Rightarrow \Delta$ by induction on the size of $\Gamma \Rightarrow \Delta$.

If the sequent $\Gamma \Rightarrow \Delta$ is saturated, then the tree consisting of one node marked by $\Gamma \Rightarrow \Delta$ is a saturation tree for $\Gamma \Rightarrow \Delta$. Otherwise, there is a formula $(A\to B)\in \Gamma \cup \Delta$. 

Suppose $\Delta= A\to B, \Delta^\prime$. Then the sequent $\Gamma \Rightarrow \Delta$ can be obtained from an unprovable sequent $\Gamma, A \Rightarrow B, \Delta^\prime$ by an application of the rule ($\mathsf{\to_{R}}$). 
In addition, the size of $\Gamma, A \Rightarrow B, \Delta^\prime$ is strictly less than the size of $\Gamma \Rightarrow \Delta$. Thus, by the induction hypothesis for $\Gamma, A \Rightarrow B, \Delta^\prime$, there exists a saturation tree $\xi$ for the sequent $\Gamma, A \Rightarrow B, \Delta^\prime$. We  see that   
\begin{gather*}
\AXC{$\xi$}
\noLine
\UIC{$\vdots$}
\noLine
\UIC{$\Gamma, A \Rightarrow B, \Delta^\prime$}
\LeftLabel{$\mathsf{\to_{R}}$}
\RightLabel{ }
\UIC{$ \Gamma \Rightarrow A\to B, \Delta^\prime$}
\DisplayProof 
\end{gather*}
is a saturation tree with the root marked by $\Gamma \Rightarrow \Delta$.

Suppose $\Gamma=  \Gamma^\prime, A\to B$. Then the sequent $\Gamma \Rightarrow \Delta$ can be obtained from $\Gamma^\prime, B \Rightarrow  \Delta$ and $\Gamma^\prime \Rightarrow A, \Delta$ by an application of the rule ($\mathsf{\to_{L}}$). Consequently, one or both of these sequents are unprovable. Note that the sizes of $\Gamma^\prime, B \Rightarrow  \Delta$ and $\Gamma^\prime \Rightarrow A, \Delta$ are strictly less than the size of $\Gamma \Rightarrow \Delta$. Hence, by the induction hypothesis, there exists a saturation tree for one or both of these sequents.  
Similarly to the previous case, we obtain a saturation tree for $\Gamma \Rightarrow \Delta$ from the given saturation tree(s) by an application of the rule ($\mathsf{\to_{L1}}$), ($\mathsf{\to_{L2}}$) or ($\mathsf{\to_{L}}$).  
\end{proof}

A \emph{refutation tree} is a tree of unprovable sequents constructed according to the rules ($\mathsf{\to_{R}}$), ($\mathsf{\to_{L}}$), ($\mathsf{\to_{L1}}$), ($\mathsf{\to_{L2}}$) and ($\boxtimes$). A refutation tree is called \emph{regular} if it contains only finitely many non-isomorphic subtrees with respect to the marking of sequents.

\begin{lemma}\label{from unprovability to refutability}
For any unprovable sequent $\Gamma \Rightarrow \Delta$, there exists a regular refutation tree with the root marked by $\Gamma \Rightarrow \Delta$. 
\end{lemma}
\begin{proof}
Assume we have an unprovable sequent $\Gamma \Rightarrow \Delta$. Let $\textit{Sub}\: (\Gamma \Rightarrow \Delta)$ be the set of all subformulas of the formulas from $\Gamma \cup \Delta$. Let $S_0$ be the set of unprovable sequents of the form $\Sigma^s, \Pi^s, \Box^+ \Pi^s \Rightarrow C$ (or $\Sigma^s, \Pi^s, \Box^+ \Pi^s \Rightarrow \Box^+C$), where $\Sigma^s \subset \textit{Sub}\:(\Gamma \Rightarrow \Delta)$, $ \Pi^s \subset \textit{Sub}\:(\Gamma \Rightarrow \Delta)$ and $ C \in \textit{Sub}\:(\Gamma \Rightarrow \Delta)$ (or $ \Box^+ C \in \textit{Sub}\:(\Gamma \Rightarrow \Delta)$). We put $S:= S_0\cup \{\Gamma \Rightarrow \Delta\} $. Notice that $S$ if finite.

Applying Lemma \ref{Saturation lemma}, for any sequent $\alpha$ from $S$, we fix a saturation tree $\xi_\alpha$ with the root marked by $\alpha$. Notice that each leaf $a$ of the saturation tree $\xi_\alpha$ is marked by a saturated unprovable sequent $\Phi_a \Rightarrow \Psi_a$, where $\Phi^s_a \subset \textit{Sub}\:(\Gamma \Rightarrow \Delta)$ and $\Psi^s_a \subset \textit{Sub}\:(\Gamma \Rightarrow \Delta)$. Since $\Phi_a \Rightarrow \Psi_a$ is unprovable, any application of the rule ($\Box$) or ($\Box^+$) that draws $\Phi_a \Rightarrow \Psi_a$ must contain an unprovable sequent among its premises. It follows that $\Phi_a \Rightarrow \Psi_a$ can be obtained from unprovable sequents by an application of the rule ($\boxtimes$). Moreover, this application is uniquely determined. 

For $\alpha\in S$, let $\delta_\alpha$ be the tree of sequents obtained from $\xi_\alpha$ by extending each leaf of $\xi_\alpha$ with the corresponding application of ($\boxtimes$). We see that all premises of all application of ($\boxtimes$) in $\delta_\alpha$ belong to $S$. Now, starting from the root of $\delta_{\Gamma \Rightarrow \Delta}$ and travelling upwards, we successively extend each premise $\alpha$ of ($\boxtimes$) with the tree $\delta_\alpha$ and define a refutation tree for $\Gamma \Rightarrow \Delta$ by co-recursion.

Since $S$ is finite, the obtained refutation tree is regular.  
\end{proof}

\begin{lemma}\label{Box+ falsification lemma}
In any regular refutation tree with the root marked by $\Gamma \Rightarrow \Box^+ C, \Delta$, there is an application of the rule ($\boxtimes$) with a premise of the form $\Theta \Rightarrow C$. 
\end{lemma}
\begin{proof}
Assume we have a regular refutation tree $\delta$ with the root marked by $\Gamma \Rightarrow \Box^+ C, \Delta$. We prove the required assertion by \emph{reductio ad absurdum}.

Suppose, in the tree $\delta$, there is no application of the rule ($\boxtimes$) with a premise of the form $\Theta \Rightarrow C$. 
If we consider any application of the rule ($\boxtimes$) from $\delta$
\begin{gather*}
\AXC{$\textsc{Prem}_1$}
\AXC{$\textsc{Prem}_2$}
\LeftLabel{$\boxtimes$}
\RightLabel{ ,}
\BIC{$ \Upsilon, \Box \Sigma, \Box^+ \Pi \Rightarrow \Box A_1, \dotsc\Box A_n, \Box^+ B_1 \dotsc \Box^+ B_m, \Box^+ C,\Lambda$}
\DisplayProof 
\end{gather*}
where the succedent of the conclusion contains $\Box^+ C$, then we see that $\textsc{Prem}_2$ must contain the premise $\Sigma^s, \Pi^s, \Box^+ \Pi^s \Rightarrow \Box^+ C$ since it can not contain the sequent $\Sigma^s, \Pi^s, \Box^+ \Pi^s \Rightarrow C$. From the side condition for ($\boxtimes$), we also see that the sequent $\Sigma^s, \Pi^s, \Box^+ \Pi^s \Rightarrow C$ is provable in $\mathsf{S}$. Also, we note that, for any application of the rule ($\mathsf{\to_{R}}$), ($\mathsf{\to_{L}}$), ($\mathsf{\to_{L1}}$) or ($\mathsf{\to_{L2}}$), the succedent of each premise contains $\Box^+ C$ whenever the succedent of the conclusion contains $\Box^+ C$. 

Now we define the tree of sequents $\delta^\prime$ from the tree $\delta$ by travelling along $\delta$ from conclusions to premises and prunning each application of the rule ($\boxtimes$) of the form
\begin{gather*}
\AXC{$\textsc{Prem}_1$}
\AXC{$\textsc{Prem}_2$}
\LeftLabel{$\boxtimes$}
\RightLabel{ }
\BIC{$ \Upsilon, \Box \Sigma, \Box^+ \Pi \Rightarrow \Box A_1, \dotsc\Box A_n, \Box^+ B_1 \dotsc \Box^+ B_m,\Box^+ C,\Lambda$}
\DisplayProof 
\end{gather*}
to 
\begin{gather*}
\AXC{$\Sigma^s, \Pi^s, \Box^+ \Pi^s \Rightarrow \Box^+ C$}
\LeftLabel{$\boxtimes^\prime$}
\RightLabel{ .}
\UIC{$\Upsilon, \Box \Sigma, \Box^+ \Pi \Rightarrow \Box A_1, \dotsc\Box A_n, \Box^+ B_1 \dotsc \Box^+ B_m,\Box^+ C, \Lambda$}
\DP
\end{gather*}
We see that the succedent of each sequent from $\delta^\prime$ contains $\Box^+ C$ and there remain no applications of the rule ($\boxtimes$) in $\delta^\prime$. In addition, since the refutation tree $\delta$ is regular, the obtained tree $\delta^\prime$ contains only finitely many non-isomorphic subtrees with respect to the marking of sequents. 

For any application of the rule ($\mathsf{\to_{L1}}$) or ($\mathsf{\to_{L2}}$) in the tree $\delta^\prime$
\begin{gather*}
\AXC{$\Phi , B \Rightarrow  \Psi$}
\LeftLabel{$\mathsf{\rightarrow_{L1}}$}
\RightLabel{ ,}
\UIC{$\Phi , A \rightarrow B \Rightarrow  \Psi$}
\DisplayProof \qquad
\AXC{$\Phi \Rightarrow  A, \Psi$}
\LeftLabel{$\mathsf{\rightarrow_{L2}}$}
\RightLabel{ ,}
\UIC{$\Phi , A \rightarrow B \Rightarrow  \Psi$}
\DisplayProof 
\end{gather*}
from the side conditions of the rules, we see that the sequent $\Phi \Rightarrow  A, \Psi$ or $\Phi , B \Rightarrow  \Psi$ is provable in $\mathsf{S}$. Also, we see that, for any transformed application of the rule ($\boxtimes$) in the tree $\delta^\prime$ 
\begin{gather*}
\AXC{$\Sigma^s, \Pi^s, \Box^+ \Pi^s \Rightarrow \Box^+ C$}
\LeftLabel{$\boxtimes^\prime$}
\RightLabel{ ,}
\UIC{$\Upsilon, \Box \Sigma, \Box^+ \Pi \Rightarrow \Box A_1, \dotsc\Box A_n, \Box^+ B_1 \dotsc \Box^+ B_m,\Box^+ C, \Lambda$}
\DP
\end{gather*}
the sequent $\Sigma^s, \Pi^s, \Box^+ \Pi^s \Rightarrow C$ is provable in $\mathsf{S}$.
Since, in the tree $\delta^\prime$, there are only finitely many (non-identical) applications of the rule ($\mathsf{\to_{L1}}$), ($\mathsf{\to_{L2}}$) or ($\boxtimes^\prime$), we have finitely many corresponding provable sequents of the form $\Phi \Rightarrow  A, \Psi$, $\Phi , B \Rightarrow  \Psi$ or $\Sigma^s, \Pi^s, \Box^+ \Pi^s \Rightarrow C$.

Now we transform each application of the rule ($\mathsf{\to_{L1}}$), ($\mathsf{\to_{L2}}$) or ($\boxtimes^\prime$) in the tree $\delta^\prime$ into an application of ($\mathsf{\to_{L}}$) or ($\Box^+$) by adding the missing premise of the form $\Phi \Rightarrow  A, \Psi$, $\Phi , B \Rightarrow  \Psi$ or $\Sigma^s, \Pi^s, \Box^+ \Pi^s \Rightarrow C$ and extending this premise with a regular $\infty$-proof. If we extend identical premises with identical regular $\infty$-proofs, then we obtain a regular $\infty$-proof with the root marked by $\Gamma \Rightarrow \Box^+ C, \Delta$. 
However, the sequent $\Gamma \Rightarrow \Box^+ C, \Delta$ occurs in the refutation tree $\delta$ and must be unprovable (by the definition of refutation tree), which is a contradiction.

Consequently, there exists an application of the rule ($\boxtimes$) in the refutation tree $\delta$ with a premise of the form $\Theta \Rightarrow C$. 
\end{proof}


\begin{lemma}\label{from refutability to invalidity}
If there is a regular refutation tree with the root marked by $\Gamma \Rightarrow \Delta$, then $\Gamma \Rightarrow \Delta$ is invalid.
\end{lemma}
\begin{proof}
Assume we have  a regular refutation tree $\delta$ with the root marked by $\Gamma \Rightarrow \Delta$. For any node $c$ of $\delta$, let us denote the sequent of the node $c$ by $\Phi_c \Rightarrow \Psi_c$. 

Now we define a Kripke frame, which will be used to invalidate the sequent $\Gamma \Rightarrow \Delta$. We denote the set of nodes of $\delta$ that are conclusions of applications of the rule ($\boxtimes$) by $W$. For $a, b \in W$, we put $a \prec b$ if $b$ is a descendant of $a$ and there is exactly one application of ($\boxtimes$) in between $a$ and $b$. Besides, we denote the transitive closure of $\prec$ by $\prec^+$. We obtain the bimodal frame $(W, \prec, \prec^+)$. For this frame, we define the truth relation by letting
\[a \vDash p \Longleftrightarrow p \in \Phi_a.\]

For a node $c $ of $\delta$ and $a \in W$, we set $a \in \sat (c)$ if and only if $a$ is a descendant of $c$ in the tree $\delta$ and there are no applications of the rule ($\boxtimes$) in between $c$ and $a$. 
We claim that, for any formula $F$ and any node $c $ of $\delta$,  
\begin{gather*}\label{Form1} 
F \in \Phi_c \Longrightarrow \forall a\in \sat(c) \; a \vDash F,\\
F \in \Psi_c \Longrightarrow \forall a\in \sat(c) \; a \nvDash F.
\end{gather*}
The claim is established by induction on $\sz (F)$. 

Suppose $F=\bot$. Since the sequent $\Phi_c \Rightarrow \Psi_c$ is unprovable (by the definition of refutation tree), we have $\bot \nin \Phi_c$. We also see that $a \nvDash \bot$ for any $a \in \sat (c)$. The aforementioned assertion holds.

Suppose $F=p$. If $p \in \Phi_c$, then $p \in \Phi_a$ for any $a \in \sat (c)$. Consequently, $a \vDash p$ from the choice of the truth relation $\vDash$. Now if $p \in \Psi_c$, then $p \in \Psi_a$ for any $a \in \sat (c)$. Since the sequent $\Phi_a \Rightarrow \Psi_a$ is unprovable, we have $p \nin \Phi_a$. It follows that $a \nvDash p$ by the definition of the truth relation $\vDash$.

Suppose $F = A \to B$. If $(A \to B) \in \Phi_c$, then, on the path from $c$ to each $a\in \sat (c)$, we can find a node $a^\prime$ such that $B \in \Phi_{a^\prime}$ or $A \in \Psi_{a^\prime}$. Notice that $a \in \sat (a^\prime)$. From the induction hypothesis, we see that $a \vDash B$ or $a \nvDash A$. Consequently, $a \vDash A \to B$. 

If $(A \to B) \in \Psi_c$, then, on the path from $c$ to each $a\in \sat (c)$, we can find a node $a^\prime$ such that $A \in \Phi_{a^\prime}$ and $B \in \Psi_{a^\prime}$. Notice that $a \in \sat (a^\prime)$. From the induction hypothesis, we see that $a \vDash A$ and $a \nvDash B$. It follows that $a \nvDash A \to B$. 

Suppose $F$ has the form $\Box A$. If $\Box A \in \Phi_c$, then $\Box A \in \Phi_a$ for any $a \in \sat (c)$. In order to show that $a \vDash \Box A$, let us consider any $b\in W$ such that $a \prec b$. Recall that there is the unique application of the rule ($\boxtimes$) in between $a$ and $b$ and $a$ is the conclusion of the application. Moreover, there is a premise $a^\prime$ of the given application such that $b \in \sat(a^\prime)$. 
We see that $A \in \Phi_{a^\prime}$ and $b \vDash A$ by the induction hypothesis for $A$. We obtain that $a \vDash \Box A$.

Now if $\Box A \in \Psi_c$, then $\Box A \in \Psi_a$ for any $a \in \sat (c)$. Recall that $a$ is the conclusion of an application of the rule ($\boxtimes$) in $\delta$. Hence there is a premise $a^\prime$ of the given application such that $A \in \Psi_{a^\prime}$. Since $\sat (a^\prime) \neq \emptyset$, there is a node $b \in \sat (a^\prime)$. By the induction hypothesis for $A$, we have $b \nvDash A$. We see that $a\prec b$, $b \nvDash A$ and $a \nvDash \Box A$.

Suppose $F = \Box^+ A$. Let us check that $a \vDash \Box^+ A$ for $a\in \sat (c)$ if $\Box^+ A \in \Phi_c$. Consider any node $a \in  \sat (c)$ and an arbitrary sequence $a=a_0 \prec a_1 \prec \dotsb \prec a_n \prec a_{n+1}$. From $\Box^+ A \in \Phi_c$, we have $\Box^+ A \in \Phi_{a_i}$ for all $i\in \{0,\dotsc, n+1\}$. We recall that $a_n$ is the conclusion of an application of the rule ($\boxtimes$) in $\delta$. Therefore there is a premise $a^\prime_n$ of the given application such that $a_{n+1} \in \sat (a^\prime_n)$. In addition, we have $A \in \Phi_{a^\prime_n}$. From the induction hypothesis for $A$, we obtain $a_{n+1}\vDash A$. Consequently, $a \vDash \Box^+ A$.

If $\Box^+ A \in \Psi_c$, then $\Box^+ A \in \Psi_a$ for any $a \in \sat (c)$. Applying Lemma \ref{Box+ falsification lemma} for $C=A$, in the subtree of $\delta$ determined by $a$, we can find a node $a^\prime$ such that $A \in \Psi_{a^\prime}$. Also, there is an application of the rule ($\boxtimes$) in between $a$ and $a^\prime$. Since $\sat (a^\prime)\neq \emptyset$, there is a node $a^{\prime\prime} \in \sat (a^\prime)$. By the induction hypothesis for $A$, we have $a^{\prime\prime} \nvDash A$. We also see that $a \prec^+ a^{\prime\prime}$. Therefore $a \nvDash \Box^+ A$.

The claim is established.

Now let $r$ be the root of $\delta$. Since $\sat (r) \neq \emptyset$, there is a node $r^\prime \in \sat (r)$. We see that $\Phi_r=\Gamma$, $\Psi_r=\Delta$, $r^\prime \vDash \bigwedge \Gamma$ and $r^\prime \nvDash \bigvee \Delta$. Thus the sequent $\Gamma \Rightarrow \Delta$ is invalid.  
\end{proof}

\begin{theorem}
Any valid sequent is provable in the sequent calculus $\mathsf{S}$. 
\end{theorem}

\begin{proof}
Assume we have a valid sequent $\Gamma \Rightarrow \Delta$. We show that the sequent $\Gamma \Rightarrow \Delta$ is provable in the sequent calculus $\mathsf{S}$ by \emph{reductio ad absurdum}. If  $\Gamma \Rightarrow \Delta$ is unprovable, then there exists a regular refutation tree with the root marked by $\Gamma \Rightarrow \Delta$ from Lemma \ref{from unprovability to refutability}. Therefore, the sequent $\Gamma \Rightarrow \Delta$ is invalid by Lemma \ref{from refutability to invalidity}, which is a contradiction. Consequently, the sequent $\Gamma \Rightarrow \Delta$ is provable in $\mathsf{S}$. 
\end{proof}

\begin{corollary}\label{completeness}
If $\mathsf{K}^+\vdash \bigwedge
\Gamma \to \bigvee\Delta$, then the sequent $\Gamma \Rightarrow \Delta$ is provable by a regular $\infty$-proof.
\end{corollary}

\section{Cyclic proofs and annotations}
In order to facilitate our prove of the realization theorem, we introduce annotated versions of sequents and inference rules of the sequent calculus $\mathsf{S}$. We also define useful finite representations of regular $\infty$-proofs called cyclic (or circular) proofs.

An \emph{annotated formula} is a formula of $\mathsf{K}^+$ in which any occurrence of a modal connective is labelled with a natural number. These labelled modal connectives are denoted by $\Box_i$ and $\Box^+_i$, where $i\in \mathbb{N}$. A modal formula is called \emph{properly annotated} if distinct occurrences of $\Box$ in it are labelled with distinct natural numbers, and the same holds for the occurrences of $\Box^+$. 

An \emph{annotated sequent} is an expression of the form $\Gamma \Rightarrow_\alpha \Delta$, where all formulas in $\Gamma$ and $\Delta$ are annotated and $\alpha$ is an annotated formula of the form $\Box^+_n C$ or an auxiliary sign $\ast$. In addition, if $\alpha$ is a formula, then the musltiset $\Delta$ must contain $\alpha$. We also require that negative occurrences of modal connectives in $\Gamma \Rightarrow_\alpha \Delta$ (i.e. in $\bigwedge \Gamma \rightarrow \bigvee \Delta$) are labelled with even natural numbers and positive ones are labelled with odd numbers. An annotated sequent $\Gamma \Rightarrow_\alpha \Delta$ is called \emph{properly annotated} if the formula $\bigwedge \Gamma \rightarrow \bigvee \Delta$ is properly annotated. Here is an example of a properly annotated sequent:
\[\Box_1 p \to r, q \to \Box^+_2 (p \to  \Box_6 \bot) \Rightarrow_{\Box^+_1 p} \Box^+_1 p, \bot .\]

Annotated versions of inference rules are defined as
\begin{gather*}
\AXC{}
\RightLabel{ ,}
\UIC{$\Gamma, p \Rightarrow_\alpha p, \Delta $}
\DP  \qquad
\AXC{}
\RightLabel{ ,}
\UIC{$\Gamma, \bot \Rightarrow_\alpha  \Delta $}
\DP 
\end{gather*}
\begin{gather*}
\AXC{$\Gamma , B \Rightarrow_\alpha  \Delta$}
\AXC{$\Gamma \Rightarrow_\alpha  A, \Delta$}
\LeftLabel{$\mathsf{\rightarrow_L}$}
\BIC{$\Gamma , A \rightarrow B \Rightarrow_\alpha  \Delta$}
\DisplayProof \;,\quad
\AXC{$\Gamma, A \Rightarrow_\alpha  B ,\Delta$}
\LeftLabel{$\mathsf{\rightarrow_R}$}
\UIC{$\Gamma \Rightarrow_\alpha  A \rightarrow B ,\Delta$}
\DisplayProof \;,
\end{gather*}
\begin{gather*}
\AXC{$A_1,\dotsc ,A_k, B_1, \dotsc, B_l, \Box^+_{j_1} B_1,\dotsc, \Box^+_{j_l} B_l \Rightarrow_\ast C$}
\LeftLabel{$\mathsf{\Box}_m$}
\UIC{$ \Upsilon, \Box_{i_1} A_1, \dotsc, \Box_{i_k} A_k, \Box^+_{j_1} B_1,\dotsc, \Box^+_{j_l} B_l \Rightarrow_\alpha \Box_m C , \Lambda$}
\DisplayProof \;,\\\\
\AXC{$\Sigma, \Pi, \Box^+_{j_1} B_1,\dotsc, \Box^+_{j_l} B_l \Rightarrow_\ast C$}
\AXC{$\Sigma, \Pi, \Box^+_{j_1} B_1,\dotsc, \Box^+_{j_l} B_l \Rightarrow_{\Box^+_n C} \Box^+_n C$}
\LeftLabel{$\Box^+_n$}
\BIC{$  \Upsilon, \Box_{i_1} A_1, \dotsc, \Box_{i_k} A_k, \Box^+_{j_1} B_1,\dotsc, \Box^+_{j_l} B_l \Rightarrow_\alpha \Box^+_n C , \Lambda$}
\DisplayProof \;,
\end{gather*}
where $\Sigma= \{A_1,\dotsc ,A_k\}$ and $\Pi=\{B_1, \dotsc, B_l\}$.

An \emph{annotated $\infty$-proof} is a (possibly infinite) tree whose nodes are marked by annotated sequents and that is constructed according to annotated versions of inference rules. Moreover, every infinite branch in it must contain a tail such that all sequents in the tail are annotated with the same subscript formula $\Box^+_n C$ and the tail intersects an application of the rule ($\Box^+_n$) on the right premise infinitely many times. An annotated $\infty$-proof is \emph{regular} if it contains only finitely many non-isomorphic subtrees with respect to annotations. Also, we call an annotated $\infty$-proof \emph{properly annotated} if its root is marked by a properly annotated sequent.



Notice that if we erase all annotations in an annotated $\infty$-proof, then the resulting tree is an ordinary $\infty$-proof. Let us prove the converse.

\begin{lemma}\label{AnnLem}
Any $\infty$-proof $\pi$ can be properly annotated. Moreover, the obtained annotated $\infty$-proof can be chosen to be regular if $\pi $ is regular.
\end{lemma}
\begin{proof}

Note that, for any application of an inference rule of $\mathsf{S}$ and any annotation of its conclusion, one can annotate its premises and obtain an application of the annotated version of the rule. However, the choice of annotations for the premises is not unique. Let us fix, for any application of an inference rule of $\mathsf{S}$, some way of propagating annotations from the conclusion of the rule to its premises. We also require that this way of propagation, when moving from the conclusion of the rule ($\Box^+$) to its right premise, preserves, whenever possible, the subscript formula $\Box^+_n C$. 

Now assume we have an $\infty$-proof $\pi$ and an arbitrary proper annotation of its root. Starting from the root, we annotate $\pi$ according to the chosen way of propagating annotations and denote the resulting tree of annotated sequents by $\pi^\prime$. 

We claim that the given tree $\pi^\prime$ is an annotated $\infty$-proof. It is sufficient to check that $\pi^\prime$ satisfies the required condition on infinite branches. Suppose there is an infinite branch in $\pi^\prime$. Then, by the definition of $\infty$-proof, this branch contains a tail that does not intersect applications of the rule ($ \Box$) and applications of the rule ($\Box^+$) on the left premise. Moreover, all applications of the rule ($\Box^+$) in the tail have the same principal formula $\Box^+ A$ disregarding annotations. Note also that the tail intersects the rule ($\Box^+$) infinitely many times. Consequently, after the first application of the rule ($\Box^+$), all left-hand sides of sequents in the tail contain the formula $\Box^+ A$  disregarding annotations. According to the chosen way of propagating annotations, from now on all annotated sequents in the tail have the same subscript formula $\Box^+_n C$ and all applications of the rule ($\Box^+$) have the same principal formula $\Box^+_n C$, where $\Box^+_n C$ is an annotated version of the formula $\Box^+ A$. Therefore, every infinite branch of $\pi^\prime$ satisfies the required condition, and $\pi^\prime$ is an annotated $\infty$-proof.


We now assume that the $\infty$-proof $\pi$ is regular, and show that $\pi^\prime$ is also regular by \emph{reductio ad absurdum}. Suppose there is an infinite sequence of pairwise non-isomorphic subtrees of $\pi^\prime$. Since $\pi^\prime$ is obtained from a regular $\infty$-proof, there are only finitely many non-isomorphic subtrees disregarding annotations in $\pi^\prime$. Therefore, there is a subsequence $(\mu_i)_{i\in \mathbb{N}}$ of the given sequence, where all members are isomorphic disregarding annotations. We see that the roots of $(\mu_i)_{i\in \mathbb{N}}$ are marked by non-identical annotated sequents obtained from a single unannotated sequent $\Gamma \Rightarrow \Delta$. However, any annotated formula occurring in $\pi^\prime$ is a subformula of the annotated sequence of the root. Consequently, there can be only finitely many non-identical annotated sequents obtained from $\Gamma \Rightarrow \Delta$ in $\pi^\prime$, which is a contradiction. We conclude that the annotated $\infty$-proof $\pi^\prime$ is regular.
\end{proof}

A \emph{cyclic annotated proof} is a pair $(\kappa, d)$, where $\kappa$ is a finite tree of annotated sequents constructed in accordance with annotated versions of inference rules of $\mathsf{S}$ and $d$ is a function with the following properties: the function $d$ is defined on the set of all leaves of $\kappa$ that are not marked by sequents of the form $\Gamma,p\Rightarrow_\alpha p,\Delta$ and $\Gamma, \bot \Rightarrow_\alpha \Delta $; the image $d(a)$ of a leaf $a$ lies on the path
from the root of $\kappa$ to the leaf $a$ and is not equal to $a$; $d(a)$ and $a$ are marked by the same sequents; all sequents on the path from $d(a)$ to $a$ have the same subscript formula $\Box^+_n C$; this path intersects an application of the rule ($\Box^+_n$) on the right premise. If the function $d$ is defined at a leaf $a$, then we say that the nodes $a$ and $d(a)$ are connected by a back-link. 

Obviously, every cyclic annotated proof can be unravelled into a regular one. We prove the converse.
\begin{lemma}\label{CyclLem}
Any regular annotated $\infty$-proof can be obtained by unravelling of a cyclic annotated proof.
\end{lemma}
\begin{proof}
Assume we have a regular annotated $\infty$-proof $\pi$. Notice that each node $a$ of this tree determines the subtree $\pi_a$ with the root $a$. Let $m$ denote the number of non-isomorphic subtrees
of $\pi$. Consider any branch $a_0, a_1, \dotsc , a_m$ in $\pi$ that starts at the root of $\pi$ and has length $m + 1$. This branch defines the sequence of subtrees $\pi_{a_0},\pi_{a_1}, \dotsc,\pi_{a_m}$. Since $\pi$ contains precisely $m$ non-isomorphic subtrees, the branch contains a pair of different nodes $b$ and $c$ determining isomorphic subtrees $\pi_b$ and $\pi_c$. Without loss of generality, assume that $c$ is farther from the root than $b$. Note that all sequents on the path form $b$ to $c$ have the same subscript formula of the form $\Box^+_n C$ and this path intersects an application of the rule ($\Box^+_n$) on the right premise, because otherwise there is an infinite branch in $\pi$ that violates the corresponding condition on infinite branches of annotated $\infty$-proofs. We cut the branch under consideration at the node $c$ and connect $c$, which has become a leaf, with $b$ by a back-link. By applying a similar operation to each of the remaining branches of length $m + 1$, we ravel the regular annotated $\infty$-proof $\pi$ into the desired cyclic annotated proof.
\end{proof}

\section{Realization theorem}
In this section, we establish the realization theorem connecting the modal logic $\mathsf{K}^+$ and the justification logic $\mathsf{J}^+$. Note that all realizations constructed in the proof
will be normal.

Let us define the forgetful translation from the language of $\mathsf{J}^+$ into the language of $\mathsf{K}^+$. Given a justification formula $A$, its forgetful translation $A^\circ$ is defined inductively by
\begin{gather*}
p^\circ :=p, \qquad \bot^\circ :=\bot, \qquad (A\to B)^\circ:= (A^\circ \to B^\circ),\\ ([ w] A)^\circ:= \Box A^\circ,\qquad ([s]_\mathsf{tc}\, A )^\circ:= \Box^+ A^\circ.  
\end{gather*}
Obviously, the forgetful translation of any theorem of $\mathsf{J}^+$ is a theorem of $\mathsf{K}^+$.
The converse statement, which we give in a slightly stronger form, is called a realization theorem. 

A justification formula $B$ is a \emph{realization of a modal formula $A$} if the formula $B$ is obtained from $A$ by replacing every occurrence of $\Box$ ($\Box^+$) in $A$ with an arbitrary justification term of the first (second) sort. The realization $B$ is called \emph{normal} if distinct negative occurrences of $\Box$ ($\Box^+$) in $A$ are replaced with distinct justification variables of the first (second) sort.

\begin{theorem}[normal realization]\label{RelThm}
For any theorem $A$ of the logic $\mathsf{K}^+$, there exists its normal realization $B$ such that $B$ has an injective proof in $\mathsf{J}^+$.
\end{theorem} 

We call a cyclic annotated proof \emph{prepared} whenever, in the given proof, every occurrence of a modal rule is labelled with an additional natural number so that different occurrences of ($\Box_m$) are labelled with different natural numbers. Also, two different occurrences of ($\Box^+_n$) are labelled with the same natural number if and only if all sequents on the shortest path connecting the right premises of these occurrences have the same subscript formula $\Box^+_n C$. We denote occurrences of ($\Box_m$) and ($\Box^+_m$) labelled with a natural number $i$ by ($\Box_{m,i}$) and ($\Box^+_{m,i}$). A function $g \colon \mathbb{N}  \to \mathbb{N}$ is called a \emph{bounding function for a prepared proof $\pi$} if, for every application of ($\Box_{m,i}$) in $\pi$, we have $i<  g(m-1)$. In addition, for every application of ($\Box^+_{n,j}$), we require that $j<  g(n)$. 

Now we extend the sets of justification variables $\mathit{JV}_1 =\{x_0, x_1, \dotsc \}$ and $\mathit{JV}_2 =\{y_0, y_1, \dotsc \}$ with provisional variables of the form $x_{m,i}$ and $y_{n,j}$.
A substitution 
\[\theta = [w_1 / x_{m_1,i_1},\dotsc , w_k/x_{m_k,i_k}, s_1 / y_{n_1,j_1},\dotsc , s_l/y_{n_l,j_l} ]\]
is called \emph{finalizing} if the terms $w_1,\dotsc, w_k$ and $s_1, \dotsc, s_l$ do not contain provisional variables. In this case, we denote the set 
\[ \{x_{m_1,i_1}, \dotsc ,x_{m_k,i_k},  y_{n_1,j_1},\dotsc , y_{n_l,j_l}\}\]
by $\mathit{Dom}(\theta)$. 
A finalizing substitution $\theta$ is called \emph{adequate for a prepared cyclic annotated proof $\pi$} if $\mathit{Dom}(\theta)= \{x_{m_1,i_1}, \dotsc ,x_{m_k,i_k},  y_{n_1,j_1},\dotsc , y_{n_l,j_l}\}$ and the finite sequences $(\Box_{m_1,i_1}), \dotsc, (\Box_{m_k,i_k})$ and $(\Box^+_{n_1,j_1}), \dotsc, (\Box^+_{n_l,j_l})$ contain precisely all annotated modal rules of $\pi$.

For an arbitrary function $g \colon  \mathbb{N}  \to \mathbb{N}$, we define the following translation of annotated modal formulas to justification ones: $p^g :=p$, $\bot^g :=\bot$, $(A\to B)^g:= (A^g \to B^g)$, $(\Box_{2m} A)^g\coloneq [x_{2m}] A^g$, $(\Box^+_{2m} A)^g\coloneq [y_{2m}] A^g$,
\begin{gather*}
(\Box_{2m+1} A)^g:= 
\begin{cases}
[x_{2m+1}] A^g, & \text{if $g(2m)=0$},\\
[x_{2m+1,0}+\dotsb + x_{2m+1,g(2m)-1}] A^g, & \text{if $g(2m)\neq 0$},
\end{cases}
\end{gather*}
\begin{gather*}
(\Box^+_{2n+1} A)^g:= 
\begin{cases}
[y_{2n+1}]_\mathsf{tc}\, A^g, & \text{if $g(2n+1)=0$},\\
[y_{2n+1,0}+\dotsb + y_{2n+1,g(2n+1)-1}]_\mathsf{tc}\, A^g, & \text{if $g(2n+1)\neq0$}.
\end{cases}
\end{gather*}


For a prepared proof $\pi = (\kappa, d)$, we denote the root of $\kappa$ by $r(\pi)$. Also, for a node $c$ of $\kappa$, by $F_c$, we denote the formula $\bigwedge \Phi_c \to \bigvee \Psi_c$, where $\Phi_c \Rightarrow_\alpha \Psi_c$ is the sequent of the node $c$. 
\begin{lemma}\label{MainLem}
Suppose $\pi$ is a prepared cyclic annotated proof of $\Gamma \Rightarrow_\alpha \Delta$ and $g$ is a bounding function for $\pi$. Then there exists a finalizing substitution $\theta$  adequate for $\pi$ such that the formula $\theta (F^g_{r(\pi)} )$ has an injective proof in $\mathsf{J}^+$.
\end{lemma}
\begin{proof}
The lemma is proved by induction on the number of nodes in $\pi=(\kappa,d)$. Note that the function $g$ will be a bounding function for all prepared cyclic annotated proofs considered below.

Case 1. If $\kappa$ consists of a single node, then $\Gamma \Rightarrow_\alpha \Delta$ has the form $\Gamma^\prime, p \Rightarrow_\alpha p, \Delta^\prime$ or  $\Gamma^\prime, \bot \Rightarrow_\alpha \Delta$. Trivially, $\mathsf{J}^+_0\vdash F^g_{r(\pi)} $. Consequently, the formula $ F^g_{r(\pi)}$ has an injective proof in $\mathsf{J}^+$. We define $\theta$ as the identity substitution.

Case 2. Suppose $\Delta =  A \to B, \Delta^\prime$ and $\pi$ has the form 
\[
\AXC{$\pi^\prime$}
\noLine
\UIC{$\vdots$}
\noLine
\UIC{$\Gamma, A \Rightarrow_\alpha  B ,\Delta^\prime$}
\LeftLabel{$\mathsf{\rightarrow_R}$}
\RightLabel{ }
\UIC{$\Gamma \Rightarrow_\alpha  A \rightarrow B ,\Delta^\prime$}
\DisplayProof
\]
for a prepared cyclic annotated proof $\pi^\prime$. By the induction hypothesis, there is a finalizing substitution $\theta^\prime$ adequate for $\pi^\prime$ such that the formula $\theta^\prime (F^g_{r(\pi^\prime)} )$ has an injective proof in $\mathsf{J}^+$. We also have $\mathsf{J}^+_0\vdash\theta^\prime (F^g_{r(\pi^\prime)}) \to \theta^\prime (F^g_{r(\pi)})$. 
Consequently, the formula $\theta^\prime (F^g_{r(\pi)})$ has an injective proof in $\mathsf{J}^+$. We see that $\theta^\prime$ is adequate for $\pi$, and we set $\theta \coloneq \theta^\prime$.

Case 3. Suppose $\Gamma = \Gamma^\prime , A \to B$ and $\pi$ has the form 
\[
\AXC{$\pi^\prime$}
\noLine
\UIC{$\vdots$}
\noLine
\UIC{$\Gamma^\prime , B \Rightarrow_\alpha  \Delta$}
\AXC{$\pi^{\prime\prime}$}
\noLine
\UIC{$\vdots$}
\noLine
\UIC{$\Gamma^\prime \Rightarrow_\alpha  A, \Delta$}
\LeftLabel{$\mathsf{\rightarrow_L}$}
\RightLabel{ ,}
\BIC{$\Gamma^\prime , A \rightarrow B \Rightarrow_\alpha  \Delta$}
\DisplayProof 
\]
where $\pi^\prime$ and $\pi^{\prime\prime}$ are prepared cyclic annotated proofs. Applying the induction hypothesis for $\pi^\prime$ and $\pi^{\prime\prime}$, we find a finalizing substitution $\theta^\prime$ adequate for $\pi^\prime$ and a finalizing substitution $\theta^{\prime\prime}$ adequate for $\pi^{\prime\prime}$ such that $\theta^\prime (F^g_{r(\pi^\prime)} )$ and $\theta^{\prime\prime} (F^g_{r(\pi^{\prime\prime})} )$ are provable in $\mathsf{J}^+$ by injective proofs. Notice that  $\mathsf{J}^+_{\mathit{cs}^\prime}\vdash \theta^\prime (F^g_{r(\pi^\prime)} )$ and $\mathsf{J}^+_{\mathit{cs}^{\prime\prime}}\vdash \theta^{\prime\prime} (F^g_{r(\pi^{\prime\prime})} )$ for some finite injective constant specifications $\mathit{cs}^\prime$ and $\mathit{cs}^{\prime\prime}$. We assume that the sets $\mathit{Con} (\mathit{cs}^\prime)$ and $\mathit{Con} (\mathit{cs}^{\prime\prime})$ are disjoint. Otherwise, we can make them disjoint by renaming the constants from $\mathit{cs}^\prime$ and modifying appropriately the substitution $\theta^\prime$. Since $\mathit{Con} (\mathit{cs}^\prime) \cap \mathit{Con} (\mathit{cs}^{\prime\prime})=\emptyset$, the set $\mathit{cs}^\prime\cup\mathit{cs}^{\prime\prime}$
is a finite injective constant specification. Moreover, $\mathsf{J}^+_{\mathit{cs}^\prime \cup\mathit{cs}^{\prime\prime}}\vdash \theta^\prime (F^g_{r(\pi^\prime)} ) \wedge \theta^{\prime\prime} (F^g_{r(\pi^{\prime\prime})} )$. We see that $ \theta^\prime (F^g_{r(\pi^\prime)} ) \wedge \theta^{\prime\prime} (F^g_{r(\pi^{\prime\prime})} )$ has an injective proof in $\mathsf{J}^+$.

Notice that $\theta^\prime \circ    \theta^{\prime\prime}=\theta^{\prime\prime}\circ \theta^\prime$ since $\mathit{Dom}(\theta^\prime)\cap \mathit{Dom}(\theta^{\prime\prime})=\emptyset$. We set $\theta\coloneq\theta^\prime\circ\theta^{\prime\prime}$. Applying the substitution $\theta$ to $\theta^\prime (F^g_{r(\pi^\prime)} ) \wedge\theta^{\prime\prime} (F^g_{r(\pi^{\prime\prime})} )$, we obtain $\theta (F^g_{r(\pi^\prime)} )\wedge\theta (F^g_{r(\pi^{\prime\prime})} )$. This formula has an injective proof in $\mathsf{J}^+$ by Lemma \ref{Substitution lemma}.
Since $\mathsf{J}^+_0\vdash \theta (F^g_{r(\pi^\prime)}) \wedge \theta (F^g_{r(\pi^{\prime\prime})} ) \to\theta (F^g_{r(\pi)} )$, the formula $\theta (F^g_{r(\pi)} )$ has an injective proof in $\mathsf{J}^+$. Besides, the substitution $\theta$ is adequate for $\pi$.

Case 4. Suppose that $\pi$ has the form
\[\AXC{$\pi^\prime$}
\noLine
\UIC{$\vdots$}
\noLine
\UIC{$A_1,\dotsc ,A_k, B_1, \dotsc, B_l, \Box^+_{j_1} B_1,\dotsc, \Box^+_{j_l} B_l \Rightarrow_\ast D$}
\LeftLabel{$\mathsf{\Box}_{m,i}$}
\RightLabel{ ,}
\UIC{$ \Upsilon, \Box_{i_1} A_1, \dotsc, \Box_{i_k} A_k, \Box^+_{j_1} B_1,\dotsc, \Box^+_{j_l} B_l \Rightarrow_\alpha \Box_m D , \Lambda$}
\DisplayProof 
\]
where $\pi^\prime$ is a prepared cyclic annotated proof. By the induction hypothesis, there is a finalizing substitution $\theta^\prime$ adequate for $\pi^\prime$ such that $\theta^\prime (F^g_{r(\pi^\prime)} )$ has an injective proof in $\mathsf{J}^+$. The formula $\theta^\prime(F^g_{r(\pi^\prime)})$ has the form 
\begin{multline*}
\theta^\prime (A^g)\wedge\dotso \wedge \theta^\prime(A^g_k)\wedge\\
\wedge \theta^\prime(B^g_1)\wedge\dotso \wedge \theta^\prime(B^g_l)\wedge [y_{j_1}]_\mathsf{tc}\, \theta^\prime(B^g_1)\wedge\dotso \wedge [y_{j_l}]_\mathsf{tc}\, \theta^\prime(B^g_l) \to \theta^\prime(D^g).
\end{multline*}
From Lemma \ref{Lifting lemma}, there is a term $h$ depending only on $\{x_{i_1}, \dotsc , x_{i_k}\}$ and $\{y_{j_1}, \dotsc , y_{j_l}\}$ such that the formula
\begin{multline}\label{form2}
[x_{i_1}]\theta^\prime (A^g)\wedge\dotso \wedge [x_{i_k}]\theta^\prime(A^g_k)\wedge\\
\wedge [y_{j_1}]_\mathsf{tc}\, \theta^\prime(B^g_1)\wedge\dotso \wedge [y_{j_l}]_\mathsf{tc}\, \theta^\prime(B^g_l) \to [h] \theta^\prime(D^g)
\end{multline}
has an injective proof in $\mathsf{J}^+$.
Note that $x_{m,i} \nin \mathit{Dom}(\theta^\prime)$, i.e. $\theta^\prime(x_{m,i})=x_{m,i}$. We put $\theta\coloneq [h/ x_{m,i}] \circ \theta^\prime$.
Applying the substitution $[h/ x_{m,i}]$ to (\ref{form2}), we obtain
\begin{multline}\label{form3}
[x_{i_1}]\theta (A^g)\wedge\dotso \wedge [x_{i_k}]\theta(A^g_k)\wedge\\
\wedge [y_{j_1}]_\mathsf{tc}\, \theta(B^g_1)\wedge\dotso \wedge [y_{j_l}]_\mathsf{tc}\, \theta(B^g_l) \to [h] \theta(D^g),
\end{multline}
which has an injective proof in $\mathsf{J}^+$ by Lemma \ref{Substitution lemma}. In addition, the formula
\[ [h] \theta(D^g)\to [\theta(x_{m,0})+\dotsb + \theta(x_{m,i-1})+ h+\theta(x_{m,i+1})+ \dotsb +\theta(x_{m,g(m-1)-1})] \theta(D^g)\]
is provable in $\mathsf{J}^+_0$,
i.e. $\mathsf{J}^+_0\vdash [h] \theta(D^g)\to \theta ((\Box_m D)^g)$. Now we see that (\ref{form3}) implies $\theta (F^g_{r(\pi)} )$ in $\mathsf{J}^+_0$. Therefore, the formula $\theta (F^g_{r(\pi)} )$ has an injective proof in $\mathsf{J}^+$. Note also that $\theta$ is a finalizing substitution adequate for $\pi$.

Case 5. Suppose that there is a leaf of $\pi$ connected by a back-link with the root. In this case, all sequents on the path from the root to the leaf have the same subscript $\alpha$, and $\alpha= \Box^+_n D$ for some formula $D$. 

Let $R$ denote the following set of nodes of $\pi=(\kappa,d)$: $b\in R$ if and only if every sequent lying on the path from the root of $\pi$ to the node $b$ has the subscript formula $\Box^+_n D$. Note that, for any $b\in R$, the sequent of the node $b$ has the form $\Gamma_b\Rightarrow_{\Box^+_n D}\Delta_b, \Box^+_n D$. We set $G_b\coloneq \bigwedge \Gamma_b\wedge \neg \bigvee \Delta_b$ and $H\coloneq \bigvee \{G_b\mid b\in R \}$. Trivially, $\mathsf{K}^+\vdash \bigwedge \Gamma_b\to \bigvee (\Delta_b\cup \{H\})$ and $\mathsf{J}^+_0\vdash (\bigwedge \Gamma_b\to \bigvee (\Delta_b\cup \{H\}))^g$. 

For any $b\in R$, we define its rank
$\mathit{rk}(b) $ as follows. We put $\mathit{rk}(b)\coloneq 0$ whenever $b$ is the conclusion of a modal rule or is a leaf of $\kappa$ that is not connected by a back-link. We set $\mathit{rk}(b)\coloneq \mathit{rk}(b^\prime)+1$ whenever $b$ is a conclusion of the rule ($\to_\mathsf{R}$) and $b^\prime$ is the corresponding premise. Analogously, $\mathit{rk}(b)\coloneq \max \{\mathit{rk}(b^\prime)+1, \mathit{rk}(b^{\prime\prime})+1\}$ if $b$ is a conclusion of the rule ($\mathsf{\to_L}$) with the premises $b^\prime$ and $b^{\prime\prime}$. If $b$ is a leaf of $\kappa$ connected by a back-link with a node $c$, then we put $\mathit{rk}(b)\coloneq \mathit{rk}(c)+1$.

Let $R_0\coloneq \{a\in R \mid \mathit{rk}(a)=0\}$. For each $a\in R_0$, we define a  finalizing substitution $\sigma_a$ and a justification term $o_a$ such that $o_a$ does not contain provisional variables, $\mathit{Dom}(\sigma_a)\cap \mathit{Dom}(\sigma_b)=\emptyset$ for any two different nodes $a$ and $b$ from $R_0$ and the formula
\begin{gather*}
\sigma_a (G_a^g)\to [o_a] (\sigma_a(D^g)\wedge \sigma_a(H^g))
\end{gather*} 
has an injective proof in $\mathsf{J}^+$. In what follows, we denote the subtree of $\kappa$ with the root $a$ by $\kappa_a$.


Suppose $a$ is a leaf of $\kappa$ and is not connected by a back-link with another node of $\kappa$. In this case, the node $a$ is marked by a sequent of the form $\Gamma^\prime_a, p \Rightarrow_{\Box^+_n D}  p, \Delta^\prime_a, \Box^+_n D $ or $\Gamma^\prime_a, \bot \Rightarrow_{\Box^+_n D} \Delta_a, \Box^+_n D$. We define $\sigma_a$ as the identity substitution and put $o_a\coloneq x_0$. We see that $\mathsf{J}^+_0\vdash \neg G^g_a $ and $\mathsf{J}^+_0\vdash \neg \sigma_a (G^g_a) $. Consequently, $\mathsf{J}^+_0\vdash \sigma_a (G_a^g)\to [o_a] (\sigma_a(D^g)\wedge \sigma_a(H^g)) $.

Suppose $a$ is the conclusion of a modal rule and $\kappa_a$ has the form
\[\AXC{$\kappa^\prime_a$}
\noLine
\UIC{$\vdots$}
\noLine
\UIC{$A_1,\dotsc ,A_k, B_1, \dotsc, B_l, \Box^+_{j_1} B_1,\dotsc, \Box^+_{j_l} B_l \Rightarrow_\ast E$}
\LeftLabel{$\mathsf{\Box}_{m,i}$}
\RightLabel{ .}
\UIC{$ \Upsilon, \Box_{i_1} A_1, \dotsc, \Box_{i_k} A_k, \Box^+_{j_1} B_1,\dotsc, \Box^+_{j_l} B_l \Rightarrow_{\Box^+_n D} \Box_m E , \Lambda, \Box^+_n D$}
\DisplayProof 
\]
Since there are no applications of the rule ($\Box$) between two nodes connected by a back-link, the tree $\kappa^\prime_a$, together with the function $d$ restricted to the leaves of $\kappa^\prime_a$, defines a prepared cyclic annotated proof $\pi^\prime_a$. Let us consider the following prepared cyclic annotated proof
\[\AXC{$\pi^\prime_a$}
\noLine
\UIC{$\vdots$}
\noLine
\UIC{$A_1,\dotsc ,A_k, B_1, \dotsc, B_l, \Box^+_{j_1} B_1,\dotsc, \Box^+_{j_l} B_l \Rightarrow_\ast E$}
\LeftLabel{$\mathsf{\Box}_{m,i}$}
\RightLabel{ ,}
\UIC{$ \Upsilon, \Box_{i_1} A_1, \dotsc, \Box_{i_k} A_k, \Box^+_{j_1} B_1,\dotsc, \Box^+_{j_l} B_l \Rightarrow_{\Box^+_n D} \Box_m E , \Lambda$}
\DisplayProof 
\]
which we denote by $\pi^\prime$. Note that $a$ is different from the root of $\pi$. Therefore, $\pi^\prime$ contains fewer nodes than $\pi$. Applying the induction hypothesis, we find a finalizing substitution $\sigma_a$ adequate for $\pi^\prime$ such that $\sigma_a (F^g_{r(\pi^\prime)} )$ has an injective proof in $\mathsf{J}^+$. We see that $\mathsf{J}^+_0\vdash \sigma_a(F^g_{r(\pi^\prime)})\rightarrow \neg \sigma_a(G^g_a) $. Hence, $\neg \sigma_a (G^g_a) $ is provable in $\mathsf{J}^+$ by an injective proof. Now we put $o_a\coloneq x_0$ and obtain that $ \sigma_a (G_a^g)\to [o_a] (\sigma_a(D^g)\wedge \sigma_a(H^g)) $  has an injective proof in $\mathsf{J}^+$.

Suppose $a$ is the conclusion of a modal rule and $\kappa_a$ has the form
\[\AXC{$\kappa^\prime_a$}
\noLine
\UIC{$\vdots$}
\noLine
\UIC{$\Sigma, \Pi, \Box^+_{j_1} B_1,\dotsc, \Box^+_{j_l} B_l \Rightarrow_\ast E$}
\AXC{$\kappa^{\prime\prime}_a$}
\noLine
\UIC{$\vdots$}
\noLine
\UIC{$\Sigma, \Pi, \Box^+_{j_1} B_1,\dotsc, \Box^+_{j_l} B_l \Rightarrow_{\Box^+_m E} \Box^+_m E$}
\LeftLabel{$\Box^+_{m,j}$}
\RightLabel{ ,} 
\BIC{$  \Upsilon, \Box_{i_1} A_1, \dotsc, \Box_{i_k} A_k, \Box^+_{j_1} B_1,\dotsc, \Box^+_{j_l} B_l \Rightarrow_{\Box^+_n D} \Box^+_m E , \Lambda, \Box^+_n D$}
\DisplayProof 
\]
where $\Sigma= \{A_1,\dotsc ,A_k\}$, $\Pi=\{B_1, \dotsc, B_l\}$ and $\Box^+_m E \neq \Box^+_n D$. Since the path between any two nodes connected by a back-link can not intersect the application ($\Box^+_{m,j}$), the trees $\kappa^\prime_a$ and $\kappa^{\prime\prime}_a$, together with the function $d$ restricted to the corresponding sets of leaves, define prepared cyclic annotated proofs $\pi^\prime_a$ and $\pi^{\prime\prime}_a$. Let us consider the following prepared cyclic annotated proof
\[\AXC{$\pi^\prime_a$}
\noLine
\UIC{$\vdots$}
\noLine
\UIC{$\Sigma, \Pi, \Box^+_{j_1} B_1,\dotsc, \Box^+_{j_l} B_l \Rightarrow_\ast E$}
\AXC{$\pi^{\prime\prime}_a$}
\noLine
\UIC{$\vdots$}
\noLine
\UIC{$\Sigma, \Pi, \Box^+_{j_1} B_1,\dotsc, \Box^+_{j_l} B_l \Rightarrow_{\Box^+_m E} \Box^+_m E$}
\LeftLabel{$\Box^+_{m,j}$}
\RightLabel{ ,} 
\BIC{$  \Upsilon, \Box_{i_1} A_1, \dotsc, \Box_{i_k} A_k, \Box^+_{j_1} B_1,\dotsc, \Box^+_{j_l} B_l \Rightarrow_{\Box^+_n D} \Box^+_m E , \Lambda$}
\DisplayProof 
\]
which we denote by $\pi^\prime$. Since $a$ is different from the root of $\pi$, the proof $\pi^\prime$ contains fewer nodes than $\pi$. By the induction hypothesis, there is a finalizing substitution $\sigma_a$ adequate for $\pi^\prime$ such that $\sigma_a (F^g_{r(\pi^\prime)} )$ has an injective proof in $\mathsf{J}^+$. Note that $\mathsf{J}^+_0\vdash \sigma_a(F^g_{r(\pi^\prime)})\rightarrow \neg \sigma_a(G^g_a) $. Therefore, $\neg \sigma_a (G^g_a) $  has an injective proof in $\mathsf{J}^+$. We put $o_a\coloneq x_0$ and obtain that $ \sigma_a (G_a^g)\to [o_a] (\sigma_a(D^g)\wedge \sigma_a(H^g)) $ is provable in $\mathsf{J}^+$ by an injective proof.

Suppose $a$ is the conclusion of a modal rule and $\kappa_a$ has the form
\[\AXC{$\kappa^\prime_a$}
\noLine
\UIC{$\vdots$}
\noLine
\UIC{$\Sigma, \Pi, \Box^+_{j_1} B_1,\dotsc, \Box^+_{j_l} B_l \Rightarrow_\ast D$}
\AXC{$\kappa^{\prime\prime}_a$}
\noLine
\UIC{$\vdots$}
\noLine
\UIC{$\Sigma, \Pi, \Box^+_{j_1} B_1,\dotsc, \Box^+_{j_l} B_l \Rightarrow_{\Box^+_n D} \Box^+_n D$}
\LeftLabel{$\Box^+_{n,j}$}
\RightLabel{ ,} 
\BIC{$  \Upsilon, \Box_{i_1} A_1, \dotsc, \Box_{i_k} A_k, \Box^+_{j_1} B_1,\dotsc, \Box^+_{j_l} B_l \Rightarrow_{\Box^+_n D} \Box^+_n D , \Lambda$}
\DisplayProof 
\]
where $\Sigma= \{A_1,\dotsc ,A_k\}$ and $\Pi=\{B_1, \dotsc, B_l\}$. We see that the tree $\kappa^\prime_a$, together with the function $d$ restricted to the set of leaves of $\kappa^\prime_a$, defines a prepared cyclic annotated proof $\pi^\prime_a$. By the induction hypothesis, there is a finalizing substitution $\sigma_a$ adequate for  $\pi^\prime_a$ such that $\sigma_a (F^g_{r(\pi^\prime_a)} )$ has an injective proof in $\mathsf{J}^+$. The formula $\sigma_a(F^g_{r(\pi^\prime_a)})$ has the form 
\begin{multline*}
\sigma_a (A^g_1)\wedge\dotso \wedge \sigma_a(A^g_k)\wedge\\
\wedge \sigma_a(B^g_1)\wedge\dotso \wedge \sigma_a(B^g_l)\wedge [y_{j_1}]_\mathsf{tc}\, \sigma_a(B^g_1)\wedge\dotso \wedge [y_{j_l}]_\mathsf{tc}\, \sigma_a(B^g_l) \to \sigma_a(D^g).
\end{multline*}
From the definition of $H$, we have
\begin{multline*}
\mathsf{J}^+_0\vdash\sigma_a (A^g_1)\wedge\dotso \wedge \sigma_a(A^g_k)\wedge\sigma_a(B^g_1)\wedge\dotso \wedge \sigma_a(B^g_l)\wedge\\
\wedge  [y_{j_1}]_\mathsf{tc}\, \sigma_a(B^g_1)\wedge\dotso \wedge [y_{j_l}]_\mathsf{tc}\, \sigma_a(B^g_l) \to \sigma_a(H^g).
\end{multline*}
Hence, the formula
\begin{multline*}
\sigma_a (A^g_1)\wedge\dotso \wedge \sigma_a(A^g_k)\wedge\sigma_a(B^g_1)\wedge\dotso \wedge \sigma_a(B^g_l)\wedge \\
\wedge [y_{j_1}]_\mathsf{tc}\, \sigma_a(B^g_1)\wedge\dotso \wedge [y_{j_l}]_\mathsf{tc}\, \sigma_a(B^g_l) \to \sigma_a(D^g) \wedge \sigma_a(H^g).
\end{multline*}
has an injective proof in $\mathsf{J}^+$. By Lemma \ref{Lifting lemma}, there is a term $o_a$ depending only on $\{x_{i_1}, \dotsc , x_{i_k}\}$ and $\{y_{j_1}, \dotsc , y_{j_l}\}$ such that the formula
\begin{multline}\label{form6}
[x_{i_1}]\sigma_a (A^g_1)\wedge\dotso \wedge [x_{i_k}]\sigma_a(A^g_k)\wedge\\
\wedge [y_{j_1}]_\mathsf{tc}\, \sigma_a(B^g_1)\wedge\dotso \wedge [y_{j_l}]_\mathsf{tc}\, \sigma_a(B^g_l) \to [o_a] (\sigma_a(D^g) \wedge \sigma_a (H^g))
\end{multline}
has an injective proof in $\mathsf{J}^+$. Note that (\ref{form6}) implies $ \sigma_a (G_a^g)\to [o_a] (\sigma_a(D^g)\wedge \sigma_a(H^g)) $ in $\mathsf{J}^+_0$.

Now the finalizing substitution $\sigma_a$ and the justification term $o_a$ are well defined for any $a\in R_0$. Moreover, for each $a\in R_0$, the formula
\begin{gather}\label{form4}
\sigma_a (G_a^g)\to [o_a] (\sigma_a(D^g)\wedge \sigma_a(H^g))
\end{gather} 
is provable in $\mathsf{J}^+_{\mathit{cs}_a}$ for some finite injective constant specification $\mathit{cs}_a$. We assume that all sets $\mathit{Con}(\mathit{cs}_a)$ are pairwise disjoint. Otherwise, we make them disjoint by renaming the constants and modifying appropriately substitutions $\sigma_a$ and terms $o_a$. Note that $\sigma_a\circ    \sigma_b=\sigma_b\circ \sigma_a$ for any two different nodes $a$ and $b$ since $\mathit{Dom}(\sigma_a)\cap \mathit{Dom}(\sigma_b)=\emptyset$. Let $\sigma$ be the composition of all substitutions $\sigma_a$ for $a\in R_0$. Obviously, $\sigma$ is finalizing. Now we put $\mathit{cs}\coloneq \bigcup \{\sigma (\mathit{cs}_a)\mid a\in R_0\}$. Since the sets $\mathit{Con}(\sigma(\mathit{cs}_a))= \mathit{Con}(\mathit{cs}_a)$ are pairwise disjoint, $\mathit{cs}$ is a finite injective constant specification. 

We claim that, for each $b\in R$, there is a justification term $v_b$ such that $v_b$ does not contain provisional variables and the formula
\begin{gather}\label{form5}
\sigma (G^g_b)\to [v_b] (\sigma(D^g)\wedge \sigma(H^g)) , 
\end{gather} 
is provable in $\mathsf{J}^+_\mathit{cs}$. We  proceed by subinduction on $\mathit{rk}(b)$. 

Case A. Suppose $ \mathit{rk}(b)=0$, i.e. $b\in R_0$. Applying $\sigma$ to (\ref{form4}), we obtain 
\begin{gather*}
 \sigma (G^g_b)\to [v_b] (\sigma(D^g)\wedge \sigma(H^g)) , 
\end{gather*}
where $v_b=o_b$. This formula is provable in $\mathsf{J}^+_{\sigma(\mathit{cs}_b)}$ by Lemma \ref{Substitution lemma}. Therefore, it is provable in $\mathsf{J}^+_{\mathit{cs}}$. 

Case B. Suppose $b$ is a leaf of $\kappa$ connected by a back-link with a node $c$.  From the induction hypothesis for $c$, the formula $\sigma (G_{c}^g)\to [v_{c}] (\sigma(D^g)\wedge \sigma(H^g)) $ is provable in $\mathsf{J}^+_{\mathit{cs}}$ for some term $v_{c}$. In addition, $v_c$ does not contain provisional variables. Note that $\sigma (G_{b}^g)$ coincides with $\sigma (G_{c}^g)$. Therefore, the formula $\sigma(G_b^g)\to [v_b] (\sigma(D^g)\wedge \sigma(H^g)) $ is provable in $\mathsf{J}^+_{\mathit{cs}}$ for $v_b\coloneq v_{c}$, and $v_b$ does not contain provisional variables. 

Case C. Suppose the tree $\kappa_b$ has the form 
\[
\AXC{$\kappa^\prime_b$}
\noLine
\UIC{$\vdots$}
\noLine
\UIC{$\Gamma_b, A \Rightarrow_{\Box^+_n D}  B ,\Delta^\prime_b, \Box^+_n D$}
\LeftLabel{$\mathsf{\rightarrow_R}$}
\RightLabel{ .}
\UIC{$\Gamma_b \Rightarrow_{\Box^+_n D}  A \rightarrow B ,\Delta^\prime_b, \Box^+_n D$}
\DisplayProof
\]
In this case, $G_b $ coincides with $\bigwedge \Gamma_b \wedge \neg \bigvee \{ A \rightarrow B\}\cup \Delta^\prime_b$. Let us denote the child of $b$ by $b^\prime$. Note that $\mathit{rk}(b^\prime)< \mathit{rk}(b)$. By the subinduction hypothesis, there is a term $v_{b^\prime}$ without occurrences of provisional variables such that the formula $\sigma (G_{b^\prime}^g)\to [v_{b^\prime}] (\sigma(D^g)\wedge \sigma(H^g)) $  is provable in $\mathsf{J}^+_{\mathit{cs}}$. Since $\mathsf{J}^+_0\vdash \sigma (G_{b}^g) \to \sigma (G_{b^\prime}^g)$, we obtain $\mathsf{J}^+_\mathit{cs}\vdash \sigma (G_b^g)\to [v_b] (\sigma(D^g)\wedge \sigma(H^g)) $ for $v_b\coloneq v_{b^\prime}$. We see that $v_b$ does not contain provisional variables. 

Case D. Suppose the tree $\kappa_b$ has the form
\[
\AXC{$\kappa^\prime_b$}
\noLine
\UIC{$\vdots$}
\noLine
\UIC{$\Gamma^\prime_b , B \Rightarrow_{\Box^+_n D}  \Delta_b, \Box^+_n D$}
\AXC{$\kappa^{\prime\prime}_b$}
\noLine
\UIC{$\vdots$}
\noLine
\UIC{$\Gamma^\prime_b \Rightarrow_{\Box^+_n D}  A, \Delta_b,\Box^+_n D$}
\LeftLabel{$\mathsf{\rightarrow_L}$}
\RightLabel{ .}
\BIC{$\Gamma^\prime_b , A \rightarrow B \Rightarrow_{\Box^+_n D}  \Delta_b, \Box^+_n D$}
\DisplayProof 
\]
In this case, $G_b $ coincides with $\bigwedge \Gamma^\prime_b\cup \{ A \rightarrow B\}  \wedge \neg \bigvee \Delta_b$. Let $b^\prime$ and $b^{\prime\prime}$ be the children of $b$. We see that $\mathit{rk}(b^\prime)< \mathit{rk}(b)$ and $\mathit{rk}(b^{\prime\prime})< \mathit{rk}(b)$. By the subinduction hypotheses for $b^\prime$ and $b^{\prime\prime}$, there are terms $v_{b^\prime}$ and $v_{b^{\prime\prime}}$ without occurrences of provisional variables such that the formulas $\sigma (G_{b^\prime}^g)\to [v_{b^\prime}] (\sigma(D^g)\wedge \sigma(H^g)) $ and $\sigma (G_{b^{\prime\prime}}^g)\to [v_{b^{\prime\prime}}] (\sigma(D^g)\wedge \sigma(H^g)) $ are provable in $\mathsf{J}^+_{\mathit{cs}}$. Since $\mathsf{J}^+_0\vdash \sigma (G_{b}^g) \to \sigma (G_{b^\prime}^g)\vee \sigma (G_{b^{\prime\prime}}^g)$, we obtain $\mathsf{J}^+_\mathit{cs}\vdash \sigma (G_b^g)\to [v_{b^\prime} + v_{b^{\prime\prime}}] (\sigma(D^g)\wedge \sigma(H^g)) $.
It remains to set $v_b\coloneq v_{b^\prime}+ v_{b^{\prime\prime}}$. 

We see that, for any $b\in R$, there is a justification term $v_b$ such that $v_b$ does not contain provisional variables and formula (\ref{form5}) is provable in $\mathsf{J}^+_{\mathit{cs}}$. The claim is checked.
 
We define $v$ as the sum of the terms $v_b$  (in any order) for $b\in R$. From Axiom (v), we have $\mathsf{J}_0^+\vdash [v_b](\sigma(D^g)\wedge \sigma (H^g)) \to [v] (\sigma(D^g)\wedge \sigma (H^g))$. Consequently, $\sigma (G_b^g)\to [v] (\sigma(D^g)\wedge \sigma (H^g))$ is provable in $\mathsf{J}^+_{\mathit{cs}}$. Since $H\coloneq \bigvee \{G_b\mid b\in R \}$, the formula  
\[\sigma (H^g)\to [v] (\sigma(D^g)\wedge \sigma (H^g)) \] 
has an injective proof in $\mathsf{J}^+$. Now, by Lemma \ref{induction rule lemma}, there is a term $t$ such that $t$ does not contain provisional variables and the formula 
\[\sigma (H^g)\to [t]_\mathsf{tc} \,\sigma(D^g)\] 
has an injective proof in $\mathsf{J}^+$. Recall that in the case under consideration there is a leaf of $\pi$ connected by a back-link with the root. From the definition of cyclic annotated proof, the path from the root of $\pi$ to this leaf intersects an application of the rule ($\Box^+_{n,j}$). Furthermore, all applications of the rule ($\Box^+$) whose right premises belong to $R$ are labelled with the same indices $n$ and $j$. Note that $y_{n,j}\nin \mathit{Dom}(\sigma)$ since $\mathit{Dom}(\sigma)= \bigcup \{\mathit{Dom}(\sigma_a)\mid a\in R_0\}$. We define the substitution $\theta$ so that the value of $\theta$ coincides with the value of $\sigma$ on every justification variable except $y_{n,j}$ and $\theta (y_{n,j})= t$. Applying $[t/y_{n,j}]$ to $\sigma (H^g)\to [t]_\mathsf{tc} \,\sigma(D^g)$, we obtain the formula $\theta (H^g)\to [t]_\mathsf{tc} \,\theta(D^g)$, which, by Lemma \ref{Substitution lemma}, has an injective proof in $\mathsf{J}^+$. In addition, the formulas $ \theta (G^g_{r(\pi)}) \to \theta (H^g)$ and 
\[ [t]_\mathsf{tc} \, \theta(D^g)\to [\theta(y_{n,0})+\dotsb + \theta(y_{n,j-1})+ t+\theta(y_{n,j+1})+ \dotsb +\theta(y_{n,g(n)-1})]_\mathsf{tc} \, \theta(D^g)\]
are provable in $\mathsf{J}^+_0$. Therefore, $\theta ((G_{r(\pi)} \to \Box^+_n D)^g)$ has an injective proof in $\mathsf{J}^+$. Notice that $\theta ((G_{r(\pi)} \to \Box^+_n D)^g)$ is equivalent to the formula $\theta (F_{r(\pi)}^g)$ in $\mathsf{J}^+_0 $. Consequently, $\theta (F_{r(\pi)}^g)$ is provable in $\mathsf{J}^+$ by an injective proof. We also see that $\theta$ is adequate substitution for $\pi$.

Case 6. Suppose the lowermost application of an inference rule in $\pi$ has the form 
\[\AXC{$\Sigma, \Pi, \Box^+_{j_1} B_1,\dotsc, \Box^+_{j_l} B_l \Rightarrow_\ast D$}
\AXC{$\Sigma, \Pi, \Box^+_{j_1} B_1,\dotsc, \Box^+_{j_l} B_l \Rightarrow_{\Box^+_n D} \Box^+_n D$}
\LeftLabel{$\Box^+_{n,j}$}
\RightLabel{ ,} 
\BIC{$  \Upsilon, \Box_{i_1} A_1, \dotsc, \Box_{i_k} A_k, \Box^+_{j_1} B_1,\dotsc, \Box^+_{j_l} B_l \Rightarrow_\alpha \Box^+_n D , \Lambda$}
\DisplayProof 
\]
where $\Sigma= \{A_1,\dotsc ,A_k\}$ and $\Pi=\{B_1, \dotsc, B_l\}$. Without loss of generality, we assume that $\alpha= \Box^+_n D$. Otherwise, we replace $\alpha$ with $\Box^+_n D$ and obtain a prepared cyclic annotated proof with the same number of nodes as the proof $\pi$, and with the same formula of the root $F_{r(\pi)}$. 


From this point on, the argument repeats what happened in Case 5. The required substitution $\theta$ is defined in exactly the same way as before. Therefore, we omit further details.
\end{proof}

\begin{proof}[Proof of Theorem \ref{RelThm}]
Assume $\mathsf{K}^+\vdash A$. There exists a regular $\infty$-proof of the sequent $\Rightarrow A$ by Corollary \ref{completeness}. Applying Lemma \ref{AnnLem} to this $\infty$-proof, we find a regular properly annotated $\infty$-proof of $\Rightarrow_\alpha B$, where $B^\circ= A$. From Lemma \ref{CyclLem}, there exists a cyclic annotated proof $\eta$ for the properly annotated sequent $\Rightarrow_\alpha B$.

Using $\eta$, we define a prepared cyclic  annotated proof $\pi$ and a bounding function $g$ for this proof $\pi$ as follows. If $\eta$ contains $k$ applications of the rule ($\Box_{2m+1}$), then we enumerate these applications starting from $0$ to $k-1$ and set $g(m-1) = k$; if $\eta$ does not contain applications of ($\Box_{2m+1}$), then we set $g(m-1) = 0$. Two applications of ($\Box^+_{2n+1}$) in the proof $\eta$ are called \emph{equivalent} if all sequents on the shortest path connecting the right premises of the applications have the same subscript formula $\Box^+_{2n+1} D$. If $\eta$ contains $l$ equivalence classes of applications of the rule ($\Box^+_{2n+1}$), then we enumerate these classes starting from $0$ to $l-1$ and set $g(n) = l$; if $\eta$ does not contain applications of ($\Box^+_{2n+1}$), then we set $g(n) = 0$. We label each occurrence of the rule ($\Box^+_{2n+1}$) from the $i$-th class by $i$.

In this way, we obtain a prepared cyclic annotated proof $\pi$ with a bounding function $g$. From Lemma \ref{MainLem}, there is a finalizing substitution $\theta$ such that the formula $\theta (B^g)$ has an injective proof in $\mathsf{J}^+$. We also see that $\theta (B^g)$ does not contain provisional variables. It remains to note that $\theta (B^g)$
is a normal realization for $A$.
\end{proof}

\subsubsection*{Acknowledgements.}
I heartily thank my wife Mariya Shamkanova for her constant and warm support. SDG.
\bibliographystyle{amsplain}
\bibliography{Collected}
\end{document}